\documentclass[12pt]{article}
\usepackage{amssymb,amsthm,hyperref,amsmath,bm,amsfonts}
\usepackage{arydshln}
\usepackage{verbatim}
\usepackage{graphicx}
\usepackage{float}

\usepackage{subfigure}
\usepackage{multirow}
\usepackage{color}
\usepackage{appendix}

\usepackage{latexsym,amsmath,amsfonts,amscd}
\usepackage{epsfig}
\usepackage{changebar}
\usepackage{pstricks}
\usepackage{pst-plot}
\usepackage{multirow}
\usepackage{subfigure}
\usepackage{placeins}
\usepackage{amssymb}
\usepackage{enumerate}

\newtheorem{theorem}{Theorem}[section]

\newtheorem{lemma}[theorem]{Lemma}
\newtheorem{prop}{Proposition}[section]
\newtheorem{definition}{Definition}[section]
\newtheorem{assumption}{Assumption}[section]
\newtheorem{remark}{Remark}[section]
\numberwithin{equation}{section}

\title{Boundary conditions for hyperbolic relaxation systems with characteristic boundaries}
\author{
Yizhou Zhou \thanks{IGPM, RWTH Aachen University, Templergraben, 55, D-52062 Aachen, Germany (zhou@igpm.rwth-aachen.de)} \qquad 
Wen-An Yong \thanks{Department of Mathematical Sciences, Tsinghua University, Beijing 100084, China;
Beijing Institute of Mathematical Sciences and Applications, Beijing 101408, China (wayong@tsinghua.edu.cn)} 
}

\date{}
\begin{document}
\maketitle{}

\begin{abstract}
This paper is concerned with initial-boundary-value problems of general multi-dimensional hyperbolic relaxation systems with characteristic boundaries. 
For the characteristic case, we redefine a Generalized Kreiss condition (GKC) which is essentially necessary to have a well-behaved relaxation limit. 
Under this characteristic GKC and a Shizuta-Kawashima-like condition, we 
derive reduced boundary conditions for the relaxation limit solving the corresponding equilibrium systems 
and justify the validity thereof.
The key of the derivation is to select an elaborate version
of the characteristic GKC by invoking the Shizuta-Kawashima-like condition. 
In contrast to the existing results, the present one does not assume that the boundary is non-characteristic for either the relaxation or equilibrium systems. In this sense, this paper completes the task in deriving reduced boundary conditions for general linear relaxation systems satisfying the structural stability condition.

\end{abstract}

\hspace{-0.5cm}\textbf{Keywords:}
\small{Hyperbolic relaxation system; Reduced boundary condition; Characteristic boundary; Kreiss condition; Shizuta-Kawashima-like condition}\\

\hspace{-0.5cm}\textbf{AMS subject classification:} \small{35L50; 76N20}

\section{Introduction}

Consider first-order partial differential equations (PDEs) with a small parameter:
\begin{eqnarray}\label{1.1}
\begin{aligned}
    U_t + \sum_{j=1}^d F_j(U)_{x_j} &= \frac{1}{\epsilon} Q(U)
\end{aligned}
\end{eqnarray}
defined in $t>0$ and $x=(x_1,x_2,...,x_d) \in \Omega \subset \mathbb{R}^n$. 
Here $U=U(x,t)\in \mathbb{R}^n$ is the unknown vector-valued function, $F_j(U)~(j=1,2,...,d)$ and $Q(U)$ are smooth functions of $U$, and $\epsilon$ is a small positive parameter called the relaxation time. Such equations are referred to as {\it relaxation systems} and describe many important non-equilibrium processes including chemically reactive flows \cite{GM}, compressible viscoelastic flows \cite{CS,Y5}, traffic flows \cite{BK,HR}.
They also arise in non-equilibrium thermo-dynamics \cite{JCL,mu,ZHYY}, the kinetic theory \cite{CFL,Ga,Le,Mi}, nonlinear optics \cite{CH}, neuroscience \cite{CBF} and so on.

For such a small-parameter problem, the main interest is to study the limit as $\epsilon$ tends to zero, the so-called zero relaxation limit. 
To see the limit, we notice that the source term in the aforementioned examples usually has (or can be easily transformed into)
the form
$$
Q(U) =
\begin{pmatrix}
  0 \\
q(U)  
\end{pmatrix}
$$
with $q(U) \in \mathbb{R}^r$ consisting of $r$ linearly independent functions of $U$. 
Accordingly, we write
$$
U = \begin{pmatrix}
    u \\[1mm]
    v
\end{pmatrix},\quad 
F_j(U) = 
\begin{pmatrix}
    f_j(u,v) \\[1mm]
    g_j(u,v)
\end{pmatrix},\quad 
Q(U) = \begin{pmatrix}
  0 \\[1mm]
q(u,v)  
\end{pmatrix}.
$$
With this partition, it is easy to see that the (formal) limit satisfies $q(u,v)=0$.
From this equation one can usually get $v=h(u)$.
Thus the limit can be determined by solving the so-called
{\it equilibrium system}
\begin{equation*}
u_t+\sum_{j=1}^d f_i(u,h(u))_{x_j} = 0.
\end{equation*}
For initial-value problems, this formal limit was justified in \cite{Y-Phd,Y1} under the structural stability condition proposed therein.
Moreover, it was shown in \cite{Y-Phd,Y1,Y3} that the stability condition holds for almost all examples mentioned above and implies the hyperbolicity of both the relaxation and equilibrium systems.
In what follows, we only consider the systems satisfying the structural stability condition.

The present work focuses on initial-boundary-value problems (IBVPs) with prescribed boundary conditions.
In this case, a fundamental task is to seek boundary conditions (called {\it reduced boundary conditions, rBCs}) for the equilibrium system. Without such boundary conditions, the equilibrium system alone is generally inadequate to determine the relaxation limit even if it exists,
which is in contrast to the initial-value problems.
To illustrate this task, we consider a simple example
\begin{eqnarray}\label{example}
\begin{aligned}
u^{\epsilon}_t + 3u^{\epsilon}_x + v^{\epsilon}_x &= 0,\\[1mm]
v^{\epsilon}_t + u^{\epsilon}_x + v^{\epsilon}_x &= -v^{\epsilon}/\epsilon
\end{aligned}
\end{eqnarray}
defined for $x \geq 0$. 
Since the coefficient matrix has two positive eigenvalues, the classical theory \cite{BS,GKO} requires two boundary
conditions
$$
u^{\epsilon}(0, t)=g(t),\quad v^{\epsilon}(0, t) = h(t)
$$
and proper initial data for \eqref{example} to solve $(u^\epsilon,v^\epsilon)$ with each fixed $\epsilon > 0$. 
It is easy to see that the formal limit $(u^0,v^0)$
satisfies $v^0=0$ and the equilibrium system
\begin{equation}\label{example-equi}
u^0_t + 3u^0_x = 0,\qquad x>0.
\end{equation}
To determine the limit $u^0$ on $x>0$, a boundary condition is indispensable for \eqref{example-equi}. 
Remark that this boundary condition should be 
determined merely by the system \eqref{example} and its prescribed initial and boundary conditions, since so is the solution sequence $(u^\epsilon,v^\epsilon)$. 
Because $u^{\epsilon}(0, t) = g(t)$ for each fixed $\epsilon$, it is natural to guess that
$u^0(0, t) = g(t)$. 
However, this is wrong if $h(t) \neq 0$ and the correct boundary condition is
$$
u^0(0, t) = g(t) +
\frac{h(t)}{3}.
$$
This example indicates that deriving the rBC is not trivial.

The above fundamental task was proposed in \cite{Y-Phd} over three decades ago  and partially resolved under a non-characteristic assumption \cite{Y-Phd,Y2}. In this pioneered work 
it was observed that the structural stability condition, together with the Uniform Kreiss condition (UKC), is not enough for the existence of relaxation limit for IBVPs. 
Note that the UKC is an essentially necessary criterion for the well-posedness of (multi-dimensional) hyperbolic systems \cite{Hi,MO}. 
To remedy this, a so-called Generalized Kreiss condition (GKC) was proposed in \cite{Y-Phd,Y2} for linear version of \eqref{1.1} with constant coefficients: 
\begin{eqnarray}\label{2.1}
\begin{aligned}
    F_j(U)=A_jU,\quad Q(U)=QU.
\end{aligned}
\end{eqnarray}
Here $A_j~(j=1,2,...,d)$ and $Q$ are constant matrices. 
Like the UKC for the well-posedness, the GKC is essentially necessary for the existence of relaxation limit. Under the GKC, the reduced boundary condition was obtained. 

The goal of the present work is to remove the non-characteristic assumption in \cite{Y2}. 
To proceed, we follow \cite{Kr,MO,Y2} and assume  
$\Omega=\{x_1>0\}$. For such a domain, the boundary $\partial\Omega = \{x_1=0\}$ is called characteristic for \eqref{1.1} with \eqref{2.1} if the coefficient matrix $A_1$ has zero eigenvalues. 
Because the boundary could be characteristic for either the relaxation system or the equilibrium system, there are the following four different cases: 
\begin{table}[H]
\centering
\small
\begin{tabular}{|c|c|c|c|c|}
\hline
&&&&\\[-1em]
Relaxation system & \quad N ~\quad & \quad C ~\quad & \quad N ~\quad & C \\
\hline
&&&&\\[-1em]
Equilibrium system & \quad N ~\quad & \quad N ~\quad & \quad C ~\quad & C \\
\hline
&&&&\\[-1em]
solved in & \cite{Y2} & \cite{ZY} & \cite{ZY3} & Present work \\
\hline
\end{tabular}
\\[3mm]
N: non-characteristic, C: characteristic
\end{table}
\noindent 
The non-characteristic assumption in \cite{Y2} means that 
the boundary is non-characteristic for both the relaxation and equilibrium systems, corresponding to the first column of the table. In \cite{ZY,ZY3}, the boundary is assumed to be non-characteristic only for one of the two systems. 
The present work allows the boundary to be characteristic for both the systems. In this sense, the non-characteristic assumptions in \cite{Y2,ZY,ZY3} are completely removed.

At this point, let us make the following remarks. (1) Characteristic IBVPs are classical \cite{MO} and the double characteristic case occurs in many situations, such as the Grad's moment system defined in the half-space $x_1>0$ \cite{LY}.
(2) Compared to the non-characteristic cases, characteristic ones possess new mathematical challenges. At least, the GKC or UKC needs to be redefined because they involve the inverse of $A_1$.
(3) It is a common practice that the theory of hyperbolic IBVPs starts with linear problems with constant coefficients defined in the half space $\{x_1>0\}$ \cite{BS,Kr,MO}. In fact, the IBVPs in a general smooth domain can be converted to finitely many similar problems in the half-space and one in the
whole space by means of a local
coordinate change and a partition of unity \cite{MO}. Therefore, the half-space is representative.
On the other hand, the linear theory has its own interest and turns out to be essential for studying nonlinear problems \cite{Majda1,Majda2}. (4) Last but not least, deriving the rBCs is a fundamental task also for studying singular limiting problems of other PDEs.

In this work, we redefine the crucial GKC 
by combining the ideas in \cite{MO,Y2}. 
Under this characteristic-GKC (c-GKC) and a Shizuta-Kawashima-like condition used in \cite{AY}, we 
derive rBCs for the relaxation limit solving the corresponding equilibrium system. 
To justify the derivation,
we show that the rBCs satisfy the (characteristic) UKC \cite{MO} for the equilibrium system. Particularly, our rBCs do not involve the characteristic modes corresponding to the zero eigenvalue for the equilibrium system. Details can be founded in Theorem \ref{theorem43}.
The key step of the derivation is to select an elaborate version of the c-GKC by invoking the Shizuta-Kawashima-like condition. Having this version, we can adapt the subtle analytical expansions of matrices in \cite{ZY3} to the present problem. 
At last, the validity of the rBC is established by using the energy method in \cite{GKO}.

Now we briefly review the existing literature on IBVPs of hyperbolic relaxation systems. Unlike ours, all of those works are about specific relaxation systems. For the Jin-Xin model \cite{JX}, there are a series of papers \cite{WX,XX,XX1,Xu} devoted to studying the relaxation limit and boundary-layer behaviors. 
About the one-dimensional Kerr-Debye model in nonlinear optics, the authors in \cite{CH} studied the relaxation limit of a specific IBVP. 
In \cite{BK}, the authors derived the rBC for a nonlinear discrete-velocity model of traffic flows by solving a boundary Riemann problem. 
For the Grad's moment closure systems with Maxwell boundary conditions, 
the GKC is verified in \cite{Zhao} for a simplified moment system and the solvability of boundary-layer equations for some flow problems is investigated in \cite{LY}. Particularly, the Shizuta-Kawashima-like condition was tacitly verified in \cite{LY}. 
The interested reader is referred to \cite{CY,Xu1,ZY4} for further related works.

This paper is organized as follows. Basic assumptions and notations are listed in Section \ref{Section2}. In Section \ref{sec3}, we present the characteristic Generalized Kreiss condition for IBVPs of relaxation systems. Section \ref{section4} is devoted to the derivation of the reduced boundary conditions. Finally, in Section \ref{section5} we prove the validity of the reduced boundary condition.

\section{Basic assumptions}\label{Section2}

In this section, we list all the key assumptions used in this paper. 
\subsection{Assumptions for the relaxation system}\label{Section2.1}

For the linear system \eqref{1.1} with constant coefficients \eqref{2.1}:
\begin{eqnarray}\label{new-2.1}
\begin{aligned}
    U_t + \sum_{j=1}^d A_j U_{x_j} &= \frac{1}{\epsilon} Q U,
\end{aligned}
\end{eqnarray}
we firstly recall the structural stability condition. 
It should be noted that the stability condition was proposed in \cite{Y-Phd,Y1} originally for nonlinear problems. 

\begin{definition}[Structural stability condition] This condition consists of the following three items:\\
(i) There is an invertible $n\times n$ matrix $P$ and an invertible $r\times r$ matrix $S$ 
such that 
$$
PQ=
\begin{pmatrix}
    0 & 0 \\[1mm]
  0 & S
\end{pmatrix}P.
$$

\noindent (ii) There exists a symmetric positive definite matrix $A_0$, called the symmetrizer, such that
$$
A_0A_j = A_j^TA_0,\qquad j=1,2,...,d.
$$ 
Here and below the superscript $T$ means the transpose of matrices or vectors.

\noindent (iii) The hyperbolic part and the source term are coupled in the following sense
$$
A_0Q+Q^TA_0\leq -P^T
\begin{pmatrix}
  0 & 0 \\[1mm]
  0 & I_r
\end{pmatrix}P.
$$
Here $I_r$ is the unit matrix of order $r$.
\end{definition}

Additionally, we also assume that the symmetrizer $A_0$ and the source term $Q$ satisfy
\begin{equation}\label{2.2}
  A_0 Q =Q^T A_0, 
\end{equation} 
This additional condition corresponds to the celebrated Onsager reciprocal relation and 
is satisfied by a large class of physically motivated relaxation models as shown in \cite{Y3}.

Under the structural stability condition and \eqref{2.2}, we may as well assume that the matrices $A_j\ (j=1,2,...,d)$ are all symmetric, the symmetrizer $A_0$ equals to $I_n$, and $Q=\text{diag}(0,S)$ with $S$ a symmetric negative definite matrix (see the discussion in Section 2 of \cite{ZY3}). 
Corresponding to the partition of $Q$, we write  
\begin{equation}\label{partA}
A_1 = 
\begin{pmatrix}
A_{11} & A_{12} \\[2mm]
A_{12}^T & A_{22}    
\end{pmatrix}, \quad 
A_j = 
\begin{pmatrix}
A_{j11} & A_{j12} \\[2mm]
A_{j12}^T & A_{j22}    
\end{pmatrix},~~j=2,...,d.
\end{equation}
Furthermore, we need 
\begin{definition}[Shizuta-Kawashima-like condition]
The kernel of matrix $Q$ does not contain the eigenvectors of $A_1$ associated with the zero eigenvalue:
\begin{equation}\label{kawa}
    \ker (A_1) \cap \ker (Q) = \{0\}.
\end{equation}       
\end{definition}

Remark that \eqref{kawa} is much weaker than the usual Shizuta-Kawashima condition in \cite{SK}:
$$
\ker \left(\sum_{j=1}^d\omega_j A_j - \lambda I_n \right) \cap \ker (Q) = \{0\}
$$
for any $\lambda\in \mathbb{R}$ and any $(\omega_1,\omega_2,...,\omega_d) \in \mathbb{R}^d\setminus \{0\}$.
This weaker version was also used in \cite{AY} for the existence of shock profiles. It is satisfied by all the three cases studied in \cite{Y2,ZY,ZY3}. Indeed, when $A_1$ is invertible as assumed in \cite{Y2,ZY3}, then \eqref{kawa} is obviously true because $\ker(A_1)=\{0\}$; when $A_{11}$ is invertible as in \cite{ZY}, it can be seen from the partition \eqref{partA} that $\ker(A_1)$ does not contain any nonzero vectors in $\ker(Q)=\mathbb{R}^{n-r} \times \{0\}$.

We end this subsection with the 
following equivalent version of the condition \eqref{kawa}. 

\begin{prop}\label{new-prop2.1}
The Shizuta-Kawashima-like condition is equivalent to the invertibility of the matrix $R_0^TQR_0$. Here $R_0\in \mathbb{R}^{n\times n_0}$ consists of $n_0$ eigenvectors of $A_1$ associated with the zero eigenvalue of multiplicity $n_0$.
\end{prop}

\begin{proof}

If $R_0^TQR_0$ is not invertible, there exists a nonzero $x\in \mathbb{R}^{n_0}$ such that 
$x^TR_0^TQR_0x = 0$. 
By the expression $Q=\text{diag}(0,S)$ with $S<0$, we deduce that 
$R_0x\in\text{ker} (Q)$. This contradicts to \eqref{kawa} since $A_1R_0x=0$.
Conversely, if there is a nonzero $x\in \mathbb{R}^n$ such that $A_1x=Qx=0$, then there exists a nonzero $y \in \mathbb{R}^{n_0}$ such that $x=R_0y$. 
Thus $R_0^TQR_0y=R_0^TQx=0$. This completes the proof.
\end{proof}

\subsection{Assumptions for the boundary condition}\label{section2.3}

For the relaxation system \eqref{new-2.1} defined in the half-space $\{x_1>0\}$, we allow the boundary $x_1=0$ to be characteristic. Namely, the symmetric matrix $A_1$ may have zero as its eigenvalue. Denote by
\begin{align*}
&n_0 = \text{the multiplicity of the zero eigenvalue of}~ A_1,\\[2mm]
&n_+ = \text{the number of positive eigenvalues of}~ A_1.
\end{align*}

According to the classical theory \cite{BS,GKO} for IBVPs of first-order hyperbolic systems, $n_+$ boundary conditions 
\begin{equation}\label{Boundary conditions}
    BU(0,\hat{x},t) = b(\hat{x},t)
\end{equation}
should be prescribed at the boundary $x_1=0$ for \eqref{new-2.1}.
Here $\hat{x} = (x_2,x_3,...,x_d)$
and $B$ is a constant $n_+\times n$-matrix of full-rank. Moreover, this boundary condition should satisfy the Uniform Kreiss Condition (UKC) \cite{Kr,MO} and 
\begin{equation}\label{2.3}
BR_0 = 0.
\end{equation}
for the well-posedness. The relation \eqref{2.3} means that the characteristic mode corresponding to the zero eigenvalue is not involved in the boundary condition. Its necessity was illustrated in \cite{MO} for the characteristic IBVPs.

In studying the relaxation limit as $\epsilon\rightarrow 0$, it was observed in \cite{Y2} that the UKC and the structural stability condition are not enough to guarantee the existence of relaxation limit for IBVPs. To remedy this, a so-called Generalized Kreiss condition (GKC) was proposed in \cite{Y2}. 
To state the GKC, we recall the following terminology introduced in \cite{Y2}. 
\begin{definition}
Let $n\times n$-matrix $M$ have precisely $k\ (0\leq k\leq n)$ stable
eigenvalues. A full-rank $n\times k$-matrix $R_M^S$
is called a right-stable matrix of $M$ if
$$
MR_M^S=R_M^SS_M,
$$
where $S_M$ is a $k\times k$ stable matrix. Similarly, we can define the right-unstable matrix $R^U_M$.
\end{definition}

For the non-characteristic case where  $A_1$ is invertible, we can define
\begin{equation}\label{M-origin}
M=M(\xi,\omega,\eta) = A_1^{-1}\left(\eta Q - \xi I_n - i \sum_{j=2}^d \omega_j A_j\right)
\end{equation}
for $\eta \geq 0$, $\xi $ a complex number satisfying $Re\xi > 0$, and $\omega=(\omega_2,...,\omega_d)\in \mathbb{R}^{d-1}$. Denote by $R_M^S(\xi,\omega,\eta)$ the right-stable matrix for $M=M(\xi,\omega,\eta)$. By Lemma 2.3 in \cite{Y2}, the matrix $M(\xi,\omega,\eta)$ has $n_+$ stable eigenvalues and thereby $R_M^S(\xi,\omega,\eta)$ is an $n\times n_+$-matrix.
With this notation, the GKC can be stated as
\begin{definition}[\textbf{Generalized Kreiss condition} \cite{Y2}]\label{GKC-origin}
    There exists a constant $c_K>0$ such that 
    $$
    |\text{det}\{BR_M^S(\xi,\omega,\eta) \}|\geq c_K\sqrt{\text{det}\{R_M^{S*}(\xi,\omega,\eta)R_M^S(\xi,\omega,\eta)\} }
    $$
    for all $\eta\geq 0$, $\omega\in \mathbb{R}^{d-1}$ and $\xi$ with $Re\xi>0$. Here the superscript $*$ means the conjugate transpose of matrices.
\end{definition}
\noindent Let us mention that the original UKC \cite{Kr} is just the GKC with $\eta \equiv 0$. 

For the characteristic case where $A_1$ is not invertible, the matrix $M=M(\xi,\omega,\eta)$ and thereby the GKC are not defined. The main goal of the present paper is to study such a case. Particularly, in Section \ref{sec3} we introduce a characteristic GKC. 

To sum up, our basic assumptions for general linear relaxation systems \eqref{new-2.1} defined in the half-space are the structural stability condition, the Shizuta-Kawashima-like condition, the two assumptions in \eqref{2.2} and \eqref{2.3}, and the characteristic Generalized Kreiss condition (to be defined in Definition \ref{DefGKC}).

\section{Characteristic Generalized Kreiss condition}\label{sec3}

Following \cite{Y2}, we consider the eigenvalue problem corresponding to the linear system \eqref{new-2.1}:
\begin{eqnarray}\label{eq3.1}
    \begin{aligned}
        & \xi \widehat{U} + A_1 \widehat{U}_{x_1} + i\sum_{j=2}^d \omega_jA_j \widehat{U} = \eta Q \widehat{U}, \\[1mm]
        & B\widehat{U}|_{x_1=0} = 0.
    \end{aligned}
\end{eqnarray}
Here $\eta \geq 0$, $\omega=(\omega_2,\omega_3,...,\omega_d)\in\mathbb{R}^{d-1}$, and $\xi$ is a complex number with $Re\xi>0$. If  \eqref{eq3.1} has a bounded solution $\widehat{U}=\widehat{U}(x_1)$ for certain parameters satisfying $\eta>0$ and $Re\xi>0$, then
$$
U^{\epsilon}(x_1,x_2,...,x_d,t) = 
\exp\left(\frac{\xi t}{\eta \epsilon} + \sum_{j=2}^d\frac{i\omega_jx_j}{\eta \epsilon}\right)\widehat{U}\left(\frac{x_1}{\eta \epsilon}\right)
$$
is the solution to the linear system in \eqref{new-2.1} with homogeneous boundary conditions:
\begin{eqnarray*}
    \begin{aligned}
        & U_t + \sum_{j=1}^d A_j U_{x_j} = \frac{1}{\epsilon} Q U, \\[1mm]
        & BU|_{x_1=0} = 0.
    \end{aligned}
\end{eqnarray*}
Since $Re\xi>0$ and $\eta>0$, this solution $U^{\epsilon}$ exponentially increases for $t>0$ as $\epsilon$ goes to zero. In order to have a well-behaved zero relaxation limit, such exponentially increasing solutions should be excluded by proposing proper constraints on the boundary conditions \eqref{Boundary conditions}.

To proceed, we recall that $R_0\in \mathbb{R}^{n\times n_0}$ satisfies $A_1R_0=0$.
Let $R_1\in \mathbb{R}^{n\times (n-n_0)}$ be a matrix such that $(R_1,R_0)$ is invertible. 
Notice that the choice of $R_0$ and $R_1$ is not unique.
Denote the inverse matrix by 
$$
(R_1,R_0)^{-1} = \begin{pmatrix}
    L_1\\ 
    L_0
\end{pmatrix}.
$$ 
Then we have the decomposition
\begin{equation}\label{R1}
I_n = R_1 L_1 + R_0 L_0.
\end{equation} 
Moreover, we need the following fact.
\begin{lemma}[Proposition A.1 in \cite{ZY}]\label{new-lemma3.1}
    Let $D$ be a symmetric positive definite matrix and $E$ be a symmetric matrix. Then the complex matrix $D +iE$ is invertible.
\end{lemma}

With these preparations, we rewrite the linear system of ordinary differential equations (ODEs) in \eqref{eq3.1}.
Notice that $R_0^TA_1=0$. We multiply $R_0^T$ on the left of \eqref{eq3.1} to obtain
$$
R_0^T G(\xi,\omega,\eta) \widehat{U} = 0\quad \text{with} \quad G(\xi,\omega,\eta) = \eta Q - \xi I_n - i\sum_{j=2}^d \omega_jA_j.
$$
Thanks to \eqref{R1}, we have
$$
R_0^T G(\xi,\omega,\eta)R_0 (L_0\widehat{U}) +  R_0^T G(\xi,\omega,\eta)R_1 (L_1\widehat{U}) = 0.
$$
Set
\begin{equation}\label{defM}
G_{kl}(\xi,\omega,\eta) = R_k^TG(\xi,\omega,\eta) R_l,\qquad k,l=0,1.   
\end{equation}
By Lemma \ref{new-lemma3.1}, it is easy to see that $G_{00}(\xi,\omega,\eta)$ is invertible. Then we have
\begin{equation}\label{eq:3.4}
L_0\widehat{U} =  -G_{00}^{-1}(\xi,\omega,\eta)G_{01}(\xi,\omega,\eta) (L_1\widehat{U}).    
\end{equation}
On the other hand, multiplying $R_1^T$ on the left of \eqref{eq3.1} yields
\begin{equation}\label{new3-5}
(R_1^TA_1R_1) (L_1\widehat{U})_{x_1} = G_{11}(\xi,\omega,\eta) (L_1\widehat{U}) + 
G_{10}(\xi,\omega,\eta)(L_0\widehat{U}).
\end{equation}
Denote $\widehat{A}_1=R_1^TA_1R_1$. From the relation \begin{equation}\label{hatA1}
\begin{pmatrix}
       R_1^T\\[2mm]
       R_0^T
    \end{pmatrix}
        A_1 
    \begin{pmatrix}
       R_1 & R_0
    \end{pmatrix}
    =\begin{pmatrix}
       \widehat{A}_1 & 0\\[2mm]
       0 & 0
    \end{pmatrix},    
\end{equation}
we see that $\widehat{A}_1\in\mathbb{R}^{(n-n_0)\times (n-n_0)}$ is an invertible matrix which has the same number of positive and negative eigenvalues as $A_1$. 
Thus we deduce from \eqref{eq:3.4} and \eqref{new3-5} that 
\begin{equation}\label{equationM}
(L_1\widehat{U})_{x_1}=M(\xi,\omega,\eta)(L_1\widehat{U}),
\end{equation}
where
\begin{equation}\label{M}
M(\xi,\omega,\eta) =\widehat{A}_1^{-1}\Big[G_{11}(\xi,\omega,\eta)-G_{10}(\xi,\omega,\eta)G_{00}^{-1}(\xi,\omega,\eta)G_{01}(\xi,\omega,\eta)\Big]    
\end{equation}
is an $(n-n_0)\times (n-n_0)$ matrix. Consequently, the equation in \eqref{eq3.1} has been rewritten as \eqref{eq:3.4} and \eqref{equationM}. Moreover, the initial condition $B\widehat{U}|_{x_1=0}=0$ in \eqref{eq3.1} can be rewritten as
\begin{equation}\label{new-3.9}
B\widehat{U}|_{x_1=0}=BR_1(L_1\widehat{U})|_{x_1=0}=0
\end{equation}
due to the decomposition \eqref{R1} and the assumption \eqref{2.3}.

To solve \eqref{equationM} with the initial condition \eqref{new-3.9}, we need the following important fact about the matrix $M(\xi,\omega,\eta)$ defined in \eqref{M}.
\begin{lemma}\label{lem:3.1}
    For any $\eta \geq 0$ and any $\xi$ with $Re\xi>0$, the matrix $M(\xi, \omega ,\eta)$ has precisely $n_+$ stable eigenvalues and $(n - n_0 - n_+)$ unstable eigenvalues.
\end{lemma}
\begin{proof}
    Firstly, we show that $M(\xi, \omega ,\eta)$ has no purely imaginary eigenvalues. Otherwise, if $M(\xi, \omega ,\eta)$ has a purely imaginary eigenvalue $i\kappa$ with $\kappa \in \mathbb{R}$, then we have
    $$
    \text{det}\{i\kappa I_{n-n_0} - M(\xi, \omega ,\eta)\} = 0.
    $$
    Since $\widehat{A}_1$ is invertible, it follows that
    $$
    \text{det}\left\{G_{11}(\xi,\omega,\eta)-G_{10}(\xi,\omega,\eta)G_{00}^{-1}(\xi,\omega,\eta)G_{01}(\xi,\omega,\eta)-i\kappa \widehat{A}_1 \right\} = 0.
    $$
    From the relation
    \begin{align*}
        &~
        \begin{pmatrix}
            G_{11}-G_{10}G_{00}^{-1}G_{01} & 0\\[2mm]
            0 & G_{00}
        \end{pmatrix} - i\kappa 
       \begin{pmatrix}
       \widehat{A}_1 & 0\\[2mm]
       0 & 0
       \end{pmatrix}
       \\[2mm]
       =&~ 
       \begin{pmatrix}
            I & -G_{10}G_{00}^{-1}\\[2mm]
            0 & I
        \end{pmatrix}
       \left[\begin{pmatrix}
       G_{11} & G_{10}\\[2mm]
       G_{01} & G_{00}
       \end{pmatrix} - i\kappa 
       \begin{pmatrix}
       \widehat{A}_1 & 0\\[2mm]
       0 & 0
       \end{pmatrix}\right]
        \begin{pmatrix}
            I & 0\\[2mm]
            -G_{00}^{-1}G_{01} & I
        \end{pmatrix},
    \end{align*}
    we see that 
    $$
    \text{det}\left\{\begin{pmatrix}
       G_{11} & G_{10}\\[2mm]
       G_{01} & G_{00}
       \end{pmatrix} - i\kappa 
       \begin{pmatrix}
       \widehat{A}_1 & 0\\[2mm]
       0 & 0
       \end{pmatrix}\right\} = 
       \text{det}\left\{G_{11} - G_{10} G_{00}^{-1} G_{01} - i\kappa \widehat{A}_1 \right\}
       \text{det}\{G_{00}\} = 0.
    $$
    By  \eqref{defM} and \eqref{hatA1}, we have
    \begin{align*}
       \begin{pmatrix}
       G_{11} & G_{10}\\[2mm]
       G_{01} & G_{00}
       \end{pmatrix} - i\kappa 
       \begin{pmatrix}
       \widehat{A}_1 & 0\\[2mm]
       0 & 0
       \end{pmatrix}
       =\begin{pmatrix}
       R_1^T\\[2mm]
       R_0^T
       \end{pmatrix}
       \left(\eta Q - \xi I - i\sum_{j=2}^d \omega_jA_j - i\kappa A_1\right)\begin{pmatrix}
       R_1 & R_0
       \end{pmatrix}.
    \end{align*}
    Since $(R_1,R_0)$ is invertible, this relation implies that 
    $$
    \text{det}\left\{\eta Q - \xi I - i\sum_{j=2}^d \omega_jA_j - i\kappa A_1 \right\} = 0.
    $$
    On the other hand, we know from Lemma \ref{new-lemma3.1} that $\eta Q - \xi I - i\sum_{j=2}^d \omega_jA_j - i\kappa A_1$ is invertible, which leads to a contradiction. Thus we conclude that $M(\xi, \omega ,\eta)$ has no purely imaginary eigenvalue.

    Since the eigenvalues of $M(\xi, \omega ,\eta)$ are continuous with respect to the parameters, the numbers of stable and unstable eigenvalues for $M(\xi, \omega ,\eta)$ are invariant. Thus it suffices to show that $M(\xi, 0,0)=-\xi \widehat{A}_1^{-1}$ has $n_+$ stable eigenvalues and $(n-n_0-n_+)$ unstable eigenvalues. Since $Re\xi>0$, the result follows from the relation \eqref{hatA1} and the fact that $A_1$ has $n_+$ positive eigenvalues, $n_0$ zero eigenvalues and $(n-n_0-n_+)$ negative eigenvalues. This completes the proof.
\end{proof}

By this lemma, the right-stable matrix $R_M^S(\xi,\omega,\eta)$ of $M(\xi,\omega,\eta)$ is an $(n-n_0)\times n_+$ matrix. 
Thus, we can easily see from \eqref{eq:3.4}, \eqref{equationM} and \eqref{new-3.9} that 
the linear system of ODEs \eqref{eq3.1} has no bounded solutions if 
$$
\text{det}\{BR_1R_M^S(\xi,\omega,\eta) \}\neq 0.
$$
Following the ideas in \cite{Kr,MO,Y2}, we introduce
\begin{definition}[\textbf{Characteristic Generalized Kreiss condition}]\label{DefGKC}
    There exists a constant $c_K>0$ such that 
    $$
    |\text{det}\{BR_1R_M^S(\xi,\omega,\eta) \}|\geq c_K\sqrt{\text{det}\{R_{M}^{S*}(\xi,\omega,\eta)R_M^S(\xi,\omega,\eta)\} }
    $$
    for all $\eta\geq 0$, $\omega\in \mathbb{R}^{d-1}$ and $\xi$ with $Re\xi>0$.
\end{definition}

\begin{remark}\label{remark3.1}
When $\eta=0$, this characteristic GKC is just the Uniform Kreiss Condition (UKC) for characteristic IBVPs \cite{MO}.
On the other hand, when $n_0=0$, the matrix $R_0$ is void and $R_1$ can be the identity matrix $I_n$. In this case, we have $\widehat{A}_1=A_1$ and 
$$
    M(\xi,\omega,\eta) = A_1^{-1}\left(\eta Q - \xi I_n - i \sum_{j=2}^d \omega_j A_j\right),  
$$
which is just the matrix in \eqref{M-origin}. Thus this characteristic GKC is an extension of the original GKC in Definition \ref{GKC-origin}.
\end{remark}

\begin{remark}\label{remark3.3}
    Notice that the choice of $R_0$ and $R_1$ is not unique. However, the characteristic GKC does not rely on the special choice. Indeed, if we have another $(\widetilde{R}_0,\widetilde{R}_1)$ such that $A_1\widetilde{R}_0=0$ and $(\widetilde{R}_1,\widetilde{R}_0)$ is invertible, then it can be expressed in terms of $(R_1,R_0)$ as 
    \begin{equation}\label{new-R0R1}
\widetilde{R}_0 = R_0D_0,\quad \widetilde{R}_1=R_1C_1+R_0C_0,        
    \end{equation}
    where
    $D_0\in\mathbb{R}^{n_0\times n_0}$ and 
    $C_1\in\mathbb{R}^{(n-n_0)\times (n-n_0)}$ are invertible, and $C_0\in\mathbb{R}^{n_0\times (n-n_0)}$. Since $A_1R_0=0$ and $A_1$ is symmetric, we see from \eqref{new-R0R1} that
    $$
    \widetilde{R}_1^TA_1\widetilde{R}_1 = C_1^T(R_1^TA_1R_1)C_1.
    $$
    Moreover, it follows from \eqref{new-R0R1} and the definition in \eqref{defM} that
    \begin{align*}
    \widetilde{G}_{11} - \widetilde{G}_{10}\widetilde{G}_{00}^{-1}\widetilde{G}_{01} 
    =&~ \widetilde{R}_1^T \left(G - G\widetilde{R}_0 (\widetilde{R}_0^TG\widetilde{R}_0)^{-1}\widetilde{R}_0^TG \right)
    \widetilde{R}_1\\[2mm]
    =&~ (R_1C_1+R_0C_0)^T \left(G - GR_0 (R_0^TGR_0)^{-1}R_0^TG \right)
    (R_1C_1+R_0C_0).   
    \end{align*}
    Notice that $\left(G - GR_0 (R_0^TGR_0)^{-1}R_0^TG \right)R_0=0$. 
    We have
    \begin{align*}
    \widetilde{G}_{11} - \widetilde{G}_{10}\widetilde{G}_{00}^{-1}\widetilde{G}_{01} 
    =&~ C_1^T R_1^T\left(G - GR_0 (R_0^TGR_0)^{-1}R_0^TG \right)
    R_1C_1 = C_1^T(G_{11} - G_{10}G_{00}^{-1}G_{01})C_1   
    \end{align*}
    and thereby
    $$
    \widetilde{M}(\xi,\omega,\eta) = (\widetilde{R}_1^TA_1\widetilde{R}_1)^{-1} (\widetilde{G}_{11} - \widetilde{G}_{10}\widetilde{G}_{00}^{-1}\widetilde{G}_{01}) = C_1^{-1}M(\xi,\omega,\eta)C_1.
    $$
    Thus the right-stable matrix for $\widetilde{M}(\xi,\omega,\eta)$ is $\widetilde{R}_{M}^S(\xi,\omega,\eta)=C_1^{-1}R_{M}^S(\xi,\omega,\eta)$. 
    Since $BR_0=0$, we deduce from \eqref{new-R0R1} that
    $$
    |\text{det}\{B\widetilde{R}_1\widetilde{R}_{M}^S(\xi,\omega,\eta) \}|
    =|\text{det}\{BR_1R_{M}^S(\xi,\omega,\eta) \}|.
    $$
    On the other hand, by using Lemma 3.3 in \cite{Y2} we can see 
    \begin{align*}
    \sqrt{\text{det}\{R_{M}^{S*}(\xi,\omega,\eta)R_M^S(\xi,\omega,\eta)\} } 
    =&~ \sqrt{\text{det}\{\widetilde{R}_{M}^{S*}(\xi,\omega,\eta)C_1^*C_1\widetilde{R}_{M}^S(\xi,\omega,\eta)\} } \\[2mm]
    \geq &~ c_0 \sqrt{\text{det}\{\widetilde{R}_{M}^{S*}(\xi,\omega,\eta)\widetilde{R}_{M}^S(\xi,\omega,\eta)\} }
    \end{align*}
    with a constant $c_0>0$. In conclusion, the characteristic GKC holds with $(\widetilde{R}_1,\widetilde{R}_0)$ if so does it with $(R_1,R_0)$. This means that the characteristic GKC is an intrinsic property of the system \eqref{new-2.1} together with the boundary condition \eqref{Boundary conditions}.
\end{remark}

\section{Reduced boundary conditions}\label{section4}

In this section, we derive the reduced boundary condition for the relaxation limit. 
Unlike those in our previous works \cite{Y2,ZY,ZY3}, both the 
coefficient matrices $A_1$ and $A_{11}$ may have zero eigenvalues. Denote
\begin{align*}
&n_{10} = \text{the multiplicity of the zero eigenvalue of}~ A_{11},\\[2mm]
&n_{1+} = \text{the number of positive eigenvalues of}~ A_{11}.
\end{align*}
Recall Section \ref{section2.3} that $n_0$ and $n_+$ represent the numbers of zero and positive eigenvalues for $A_1$, respectively. 
This section has four subsections and we start with 

\subsection{Asymptotic expansions }\label{Section4.1}
In this subsection, we review the asymptotic expansions in \cite{ZY3}.
Set
$$
y=\frac{x_1}{\epsilon}, \quad z=\frac{x_1}{\sqrt{\epsilon}}
$$ 
and consider an approximate solution of the form
\begin{align}
U_\epsilon(x_1,\hat{x},t)=&
\begin{pmatrix}
  \bar{u} \\[2mm]
  \bar{v}
\end{pmatrix}(x_1,\hat{x},t)+
\begin{pmatrix}
  \mu_0 \\[2mm]
  \nu_0
\end{pmatrix}(y,\hat{x},t)+
\begin{pmatrix}
  \mu_1 \\[2mm]
  \nu_1
\end{pmatrix}(z,\hat{x},t)
+\sqrt{\epsilon}
\begin{pmatrix}
  \mu_2 \\[2mm]
  \nu_2
\end{pmatrix}(z,\hat{x},t).\label{3.1}
\end{align}
Here the partition is the same as that in \eqref{partA}, the first term
$(\bar{u},\bar{v})$ is the outer solution, and the other terms are boundary-layer corrections satisfying 
\begin{equation}\label{correction:inf}
    (\mu_0,\nu_0)(\infty,\hat{x},t)=(\mu_1,\nu_1)(\infty,\hat{x},t)=(\mu_2,\nu_2)(\infty,\hat{x},t)=0.
\end{equation}
It is not difficult to see that the outer solution solves the equilibrium system 
\begin{gather}
\bar{u}_{t}+A_{11}\bar{u}_{x_1}+\sum_{j=2}^dA_{j11}\bar{u}_{x_j}=0,\label{3.2}\\[2mm]
\bar{v}=0.\label{3.3}
\end{gather}

To derive equations for the boundary-layer correction terms, we denote 
by $\mathcal{L}$ the differential operator
\begin{equation*}
	\mathcal{L}(U):=\partial_tU+\sum_{j=1}^d A_j\partial_{x_j}U-QU/\epsilon
\end{equation*}
and write
\begin{equation}\label{Residual}
\mathcal{L}(U_\epsilon)=\mathcal{L}(U_{outer})+\mathcal{L}(U_{layer}),  
\end{equation}
where
$$
U_{outer} = 
\begin{pmatrix}
    \bar{u}\\[1mm]
    \bar{v}
\end{pmatrix},\qquad 
U_{layer} = 
\begin{pmatrix}
    \mu_0+\mu_1+\sqrt{\epsilon}\mu_2\\[1mm]
    \nu_0+\nu_1+\sqrt{\epsilon}\nu_2
\end{pmatrix}.
$$
By \eqref{3.2} and \eqref{3.3}, we have
\begin{align}
\mathcal{L}(U_{outer})=&~\partial_{t}\begin{pmatrix}
    \bar{u}\\[2mm]
    \bar{v}
\end{pmatrix} +
\sum\limits_{j=1}^d
\begin{pmatrix}
   A_{j11} & A_{j12}\\[2mm]
   A_{j12}^T & A_{j22}
\end{pmatrix}
\partial_{x_j}\begin{pmatrix}
    \bar{u}\\[2mm]
    \bar{v}
\end{pmatrix} 
  -\frac{1}{\epsilon}
\begin{pmatrix}
   0 & 0\\[2mm]
   0 & S
\end{pmatrix}
\begin{pmatrix}
    \bar{u}\\[2mm]
    \bar{v}
\end{pmatrix} \nonumber \\[2mm]
=&~
\begin{pmatrix}
    0 \\
  \sum\limits_{j=1}^d A_{j12}^T\partial_{x_j}\bar{u}
\end{pmatrix}. \label{residual-1}
\end{align}
On the other hand, we compute:
\begin{align}
\mathcal{L}(U_{layer})=&~\partial_t
\begin{pmatrix}
    \mu_0+\mu_1+\sqrt{\epsilon}\mu_2\\[2mm]
    \nu_0+\nu_1+\sqrt{\epsilon}\nu_2
\end{pmatrix}+\sum\limits_{j=2}^d \begin{pmatrix}
   A_{j11} & A_{j12}\\[2mm]
   A_{j12}^T & A_{j22}
\end{pmatrix}
\partial_{x_j}
\begin{pmatrix}
    \mu_0+\mu_1+\sqrt{\epsilon}\mu_2\\[2mm]
    \nu_0+\nu_1+\sqrt{\epsilon}\nu_2
\end{pmatrix} \nonumber \\[2mm]
& + \begin{pmatrix}
   A_{11} & A_{12}\\[2mm]
   A_{12}^T & A_{22}
\end{pmatrix}\bigg[\frac{1}{\epsilon}
~\partial_y\begin{pmatrix}
    \mu_0\\[2mm]
    \nu_0
\end{pmatrix} +\frac{1}{\sqrt{\epsilon}}
~\partial_z\begin{pmatrix}
    \mu_1\\[2mm]
    \nu_1
\end{pmatrix} + 
\partial_z\begin{pmatrix}
    \mu_2\\[2mm]
    \nu_2
\end{pmatrix} \bigg] \nonumber \\[2mm]
&-\frac{1}{\epsilon}
\begin{pmatrix}
   0 & 0\\[2mm]
   0 & S
\end{pmatrix}\begin{pmatrix}
    \mu_0+\mu_1+\sqrt{\epsilon}\mu_2\\[2mm]
    \nu_0+\nu_1+\sqrt{\epsilon}\nu_2
\end{pmatrix}. \label{residual-2}
\end{align}

The asymptotic solution $U_\epsilon$ should make $\mathcal{L}(U_\epsilon)$
as small as possible. Thus we 
let the coefficients of $\epsilon^{-1}, \epsilon^{-1/2}$, $\epsilon^{0}$ in \eqref{residual-2} be zero and obtain
\begin{eqnarray}\label{eq:residualbl}
\begin{aligned}
O\left(\epsilon^{-1}\right):&\quad \begin{pmatrix}
A_{11} & A_{12} \\[2mm]
A_{12}^T & A_{22}
\end{pmatrix}
\partial_y\begin{pmatrix}
\mu_0\\[2mm]
\nu_0
\end{pmatrix}
=\begin{pmatrix}
0 & 0 \\[2mm]
0 & S
\end{pmatrix}
\begin{pmatrix}
\mu_0+\mu_1\\[2mm]
\nu_0+\nu_1
\end{pmatrix}, \\[3mm]
O\left(\epsilon^{-1/2}\right):&\quad
\begin{pmatrix}
A_{11} & A_{12} \\[2mm]
A_{12}^T & A_{22}
\end{pmatrix}
\partial_z\begin{pmatrix}
\mu_1\\[2mm]
\nu_1
\end{pmatrix}
=\begin{pmatrix}
0 & 0 \\[2mm]
0 & S
\end{pmatrix}
\begin{pmatrix}
\mu_2\\[2mm]
\nu_2
\end{pmatrix}, \\[2mm]
O\left(\epsilon^{0}\right):&\quad 
~\partial_t\mu_{1}+A_{11}\partial_z\mu_{2}+A_{12}\partial_z\nu_{2} +\sum\limits_{j=2}^d\big[A_{j11}\partial_{x_j} \mu_1 +A_{j12}\partial_{x_j} \nu_1\big]=0.
\end{aligned}
\end{eqnarray}
Here, for the coefficient of $\epsilon^0$, only the first $(n-r)$ components and the z-dependent terms are considered. 
In the first equation of \eqref{eq:residualbl}, the unknowns $(\mu_0,\nu_0)$ and $(\mu_1,\nu_1)$ should be considered separately since $y$ and $z$ are independent variables, which gives
\begin{align}
&~~\nu_1=0, \label{eq:blcorrection11}\\[2mm]
&\begin{pmatrix}
A_{11} & A_{12} \\[2mm]
A_{12}^T & A_{22}
\end{pmatrix}
\partial_y
\begin{pmatrix}
\mu_0\\[2mm]
\nu_0
\end{pmatrix}=\begin{pmatrix}
0 & 0\\[2mm]
0 & S
\end{pmatrix}
\begin{pmatrix}
\mu_0\\[2mm]
\nu_0
\end{pmatrix}\label{eq:blcorrection12}.
\end{align}
By \eqref{eq:blcorrection11}, the second equation in \eqref{eq:residualbl} becomes
\begin{align}
&~A_{11}\partial_z\mu_1 = 0, \label{eq:blcorrection13}\\[2mm] 
&~A_{12}^T\partial_z\mu_1=S\nu_2. \label{eq:blcorrection14}
\end{align}
If the matrix $A_{11}$ is invertible, we obtain $\mu_1=0$ from \eqref{eq:blcorrection13} and \eqref{correction:inf}. This corresponds to the non-characteristic case in \cite{Y2} and there is no boundary-layer of scale $O(\sqrt{\epsilon})$.

When $A_{11}$ is not invertible, we recall that $A_{11}$ is symmetric
and introduce an orthonormal matrix $(P_1,P_0)$ satisfying
\begin{equation}\label{A11-decomposition}
    \begin{pmatrix}
        P_1^T\\[2mm]
        P_0^T
    \end{pmatrix}
    \begin{pmatrix}
        P_1 & 
        P_0
    \end{pmatrix} = I_{n-r},\qquad 
    \begin{pmatrix}
        P_1^T\\[2mm]
        P_0^T
    \end{pmatrix}
    A_{11}
    \begin{pmatrix}
        P_1 & 
        P_0
    \end{pmatrix}
    =\begin{pmatrix}
        \Lambda_1 & 0\\[2mm]
        0 & 0
    \end{pmatrix}.
\end{equation}
Here $\Lambda_1$ is an $(n-r-n_{10}) \times (n-r-n_{10})$ invertible diagonal matrix, $P_1$ and $P_0$ are $(n-r)\times (n-r-n_{10})$ and $(n-r)\times n_{10}$ matrices respectively.
Then we multiply $P_1^T$ on the left of \eqref{eq:blcorrection13} and use \eqref{correction:inf} to obtain \begin{equation}\label{P1mu}
    P_1^T\mu_1=0,
\end{equation} 
which means 
$$
\mu_1 = (P_1P_1^T+P_0P_0^T)\mu_1 = P_0(P_0^T\mu_1).
$$ 
By using \eqref{eq:blcorrection11} and \eqref{eq:blcorrection14}, we multiply $P_0^T$ on the last equation of \eqref{eq:residualbl} to obtain
\begin{equation}\label{P0mu}
\partial_t(P_0^T\mu_1)+\left[P_0^TA_{12}S^{-1}A_{12}^TP_0\right]\partial_{zz}(P_0^T\mu_1)+\sum\limits_{j=2}^d (P_0^TA_{j11}P_0) \partial_{x_j} (P_0^T\mu_1)=0.
\end{equation}
On the other hand, multiplying $P_1^T$ on the last equation of \eqref{eq:residualbl} yields
\begin{equation*}
P_1^TA_{11}\partial_z \mu_2 + P_1^TA_{12}\partial_z \nu_2+P_1^T\sum\limits_{j=2}^d A_{j11}\partial_{x_j} \mu_1 =0.
\end{equation*}
Moreover, by \eqref{A11-decomposition} and \eqref{eq:blcorrection14} we get
\begin{equation}\label{eq:blcorrection16}
P_1^T\mu_2 = -\Lambda_1^{-1}\left[ P_1^TA_{12}S^{-1}A_{12}^T\partial_z \mu_1 - \sum\limits_{j=2}^d P_1^TA_{j11}\int_z^{\infty}\partial_{x_j} \mu_1\right].
\end{equation}
Note that it is not necessary for our purpose to determine $\mu_2$ itself.

In summary, we have
\begin{eqnarray*}
\left\{{\begin{array}{*{20}l}
  \text{\eqref{3.3}}~\Rightarrow~\bar{v}=0 ,\\[3mm]
  \text{\eqref{eq:blcorrection11}}~\Rightarrow~\nu_1=0,\\[3mm]
  \text{\eqref{P1mu}}~\Rightarrow ~P_1^T\mu_1=0,
\end{array}}\right.\qquad\left\{{\begin{array}{*{20}l}
  \text{\eqref{3.2}}~\Rightarrow~\bar{u},\\[3mm]
  \text{\eqref{eq:blcorrection12}}~\Rightarrow~(\mu_0,\nu_0),\\[3mm]
  \text{\eqref{P0mu}}~\Rightarrow~P_0^T\mu_1,
\end{array}}\right.\qquad
\left\{{\begin{array}{*{20}l}
  \text{\eqref{eq:blcorrection16}}\Rightarrow~P_1^T\mu_2 ,\\[3mm]
  \text{\eqref{eq:blcorrection14}}\Rightarrow~\nu_2.
\end{array}}\right.
\end{eqnarray*}
Note that the second set of equations \eqref{3.2}, \eqref{eq:blcorrection12} and \eqref{P0mu} are differential equations while others are algebraic relations. In particular, \eqref{3.2} is the equilibrium system governing the relaxation limit. By our basic assumptions, the equilibrium system is symmetrizable hyperbolic and the coefficient matrix $A_{11}$ has $n_{1+}$ positive eigenvalues. From the classical theory of hyperbolic PDEs \cite{GKO,BS}, $n_{1+}$ boundary conditions should be given to solve the equilibrium system \eqref{3.2}.

\subsection{Boundary-layer equations}

In order to derive the $n_{1+}$ boundary conditions for the equilibrium system, we turn to the boundary-layer equations \eqref{eq:blcorrection12} and \eqref{P0mu}.
Note that the coefficient matrix $A_1$ in \eqref{eq:blcorrection12} may not be invertible which differs from our previous work \cite{ZY3}.

Observe that the boundary-layer equation \eqref{eq:blcorrection12} 
for $(\mu_0,\nu_0)$ is just 
the equation in \eqref{eq3.1} with $\xi=0$, $\omega=0$, and $\eta=1$. 
Thus we use the results in Section \ref{sec3} to obtain 
\begin{equation}\label{relation0-1}
L_0\begin{pmatrix}
    \mu_0\\[1mm]
    \nu_0
\end{pmatrix}=-G_{00}^{-1}(0,0,1)G_{01}(0,0,1)L_1\begin{pmatrix}
    \mu_0\\[1mm]
    \nu_0
\end{pmatrix}
\end{equation}
and 
\begin{equation}\label{M001}
L_1\begin{pmatrix}
    \mu_0\\[1mm]
    \nu_0
\end{pmatrix}_y = M(0,0,1)L_1\begin{pmatrix}
    \mu_0\\[1mm]
    \nu_0
\end{pmatrix},
\end{equation}
which correspond to those in \eqref{eq:3.4} and \eqref{equationM}.
Notice that
$$
-G_{00}^{-1}(0,0,1)G_{01}(0,0,1)=
-(R_0^TQR_0)^{-1} R_0^TQR_1
$$ 
and 
$M(0,0,1)= \widehat{A}_1^{-1}\widehat{Q}$ with
\begin{equation}\label{hatA-hatQ}
    \widehat{A}_1 = R_1^TA_1R_1,\qquad \widehat{Q} = R_1^TQR_1 - R_1^TQR_0 (R_0^TQR_0)^{-1} R_0^TQR_1.
\end{equation}
As in Section \ref{sec3}, once the ODE system \eqref{M001} is solved, the relation \eqref{relation0-1} can be used to determine 
$L_0\begin{pmatrix}
    \mu_0\\
    \nu_0
\end{pmatrix}$. Consequently, we can obtain the correction term
$(\mu_0,\nu_0)$ as
\begin{equation}\label{mu0nu0-relation}
\begin{pmatrix}
    \mu_0\\[1mm]
    \nu_0
\end{pmatrix}=R_0L_0\begin{pmatrix}
    \mu_0\\[1mm]
    \nu_0
\end{pmatrix}+R_1L_1\begin{pmatrix}
    \mu_0\\[1mm]
    \nu_0
\end{pmatrix}.
\end{equation}

To solve \eqref{M001}, we analyze the matrices in \eqref{hatA-hatQ}. Firstly, we select an elaborate matrix $R_1$ as follows. Recall Remark \ref{remark3.3} that the choice of $R_1$ is not unique.

\begin{prop}\label{prop4.1}
Partition the matrix $R_0$ as
$$
R_0=
\begin{pmatrix}
R_{01}\\[2mm]
R_{02}
\end{pmatrix},\quad 
R_{01}\in \mathbb{R}^{(n-r)\times n_0},\quad 
R_{02}\in \mathbb{R}^{r\times n_0}.
$$ 
Under the Shizuta-Kawashima-like condition, it follows that $r\geq n_0$ and the matrix $R_{02}$ is of full-rank. Let $R_{02}^{\perp}\in \mathbb{R}^{r\times (r-n_0)}$ be a matrix such that $(R_{02},R_{02}^{\perp})$ is invertible. Then $(R_0,R_1)$ is invertible with
\begin{equation}\label{R0R1}
R_1 = 
\begin{pmatrix}
I_{n-r} & 0\\[2mm]
0 & R_{02}^{\perp}
\end{pmatrix}.
\end{equation}
\end{prop}

\begin{proof}
With the above partition of $R_0$, we compute
\begin{equation*}
R_0^TQR_0 = 
\begin{pmatrix}
R_{01}^T &
R_{02}^T
\end{pmatrix}
\begin{pmatrix}
0 & 0\\[2mm]
0 & S
\end{pmatrix}
\begin{pmatrix}
R_{01}\\[2mm]
R_{02}
\end{pmatrix} = R_{02}^TSR_{02}.
\end{equation*}
Thanks to Proposition \ref{new-prop2.1}, $R_{02}^TSR_{02}$ is invertible. Since $S$ is symmetric negative definite, it follows that $r\geq n_0$ and $R_{02}$ is of full-rank. 
The invertibility of 
$$
(R_1,R_0) = \begin{pmatrix}
I_{n-r} & 0 & R_{01}\\[2mm]
0 & R_{02}^{\perp} & R_{02}
\end{pmatrix}
$$
is obvious. 
This completes the proof. 
\end{proof}

\begin{remark}
In what follows, we focus on the case $r>n_0$ unless otherwise specified. For $r=n_0$, the matrix $R_{02}^{\perp}$ is void and similar conclusions can be obtained easily.
\end{remark}

Having $R_1$ selected above, we compute $\widehat{A}_1$ and $\widehat{Q}$ defined in \eqref{hatA-hatQ}:
\begin{align*}
    \widehat{A}_1 =R_1^TA_1R_1=\begin{pmatrix}
I_{n-r} & 0\\[2mm]
0 & (R_{02}^{\perp})^T
\end{pmatrix}\begin{pmatrix}
        A_{11} & A_{12} \\[2mm]
        A_{12}^T & A_{22} 
    \end{pmatrix}\begin{pmatrix}
I_{n-r} & 0\\[2mm]
0 & R_{02}^{\perp}
\end{pmatrix}=\begin{pmatrix}
        A_{11} & \widehat{A}_{12} \\[2mm]
        \widehat{A}_{12}^T & \widehat{A}_{22} 
    \end{pmatrix}
\end{align*}
with
\begin{equation}\label{hatA12A22}
    \widehat{A}_{12} = A_{12}R_{02}^{\perp},\qquad 
    \widehat{A}_{22} = (R_{02}^{\perp})^TA_{22}R_{02}^{\perp}
\end{equation}
and
\begin{align}
\widehat{Q}=&~
R_1^T[Q - QR_0(R_0^TQR_0)^{-1}R_0^TQ]R_1 \nonumber \\[2mm]
= &~
R_1^T\left[\begin{pmatrix}
0 & 0 \\[2mm]
0 & S
\end{pmatrix}
-
\begin{pmatrix}
0 \\[2mm]
SR_{02}
\end{pmatrix}\left(R_{02}^TSR_{02}\right)^{-1}
\begin{pmatrix}
0 & R_{02}^TS
\end{pmatrix}\right]R_1 \nonumber \\[2mm]
=&\begin{pmatrix}
I_{n-r} & 0\\[2mm]
0 & (R_{02}^{\perp})^T
\end{pmatrix}\begin{pmatrix}
0 & 0 \\[2mm]
0 & S - SR_{02}\left(R_{02}^TSR_{02}\right)^{-1}R_{02}^TS
\end{pmatrix}\begin{pmatrix}
I_{n-r} & 0\\[2mm]
0 & R_{02}^{\perp}
\end{pmatrix} \nonumber \\[2mm]
=&
\begin{pmatrix}
0 & 0\\[2mm]
0 & \widehat{S}
\end{pmatrix} \label{new:hatQ}
\end{align}
with
\begin{equation}\label{eq24}
    \widehat{S} = (R_{02}^{\perp})^T\Big(S-SR_{02}(R_{02}^TSR_{02})^{-1}R_{02}^TS\Big)R_{02}^{\perp}.
\end{equation}
Since $(R_{02},R_{02}^{\perp})$ is invertible and $S$ is symmetric negative definite, it is not difficult to show that $\widehat{S}$ is also symmetric negative definite for $r>n_0$.

Having the block-diagonal form of $\widehat{Q}$ in \eqref{new:hatQ}, the solution of the boundary-layer equation \eqref{M001} can be attributed to the case studied in \cite{ZY3}.  
\begin{prop}[Propositions 3.2 and 3.3 in \cite{ZY3}]\label{prop2.1}
Under the basic assumptions in Section \ref{Section2.1}, 
we have the following conclusions with notations given in \eqref{A11-decomposition}, \eqref{hatA12A22}, and \eqref{eq24}.
\begin{itemize}
    \item[(i)] $r-n_0\geq n_{10}$ and the matrix $$
    K:=\widehat{A}_{12}^TP_0  \in\mathbb{R}^{(r-n_0)\times n_{10}}$$ is of full-rank.
    \item[(ii)] The solution to  \eqref{M001} can be expressed as 
\begin{equation}\label{eq:boundary-layer-correction}
L_1\begin{pmatrix}
    \mu_0\\[1.5mm]
    \nu_0
\end{pmatrix}=\begin{pmatrix}
N\\[1mm]
\widetilde{K}
\end{pmatrix}w,
\end{equation}
where $w\in\mathbb{R}^{r-n_0-n_{10}}$ satisfies the ordinary differential equation 
\begin{align}\label{ODEw}
\partial_y w = \big(\widetilde{K}^TX\widetilde{K}\big)^{-1}\big(\widetilde{K}^T \widehat{S}\widetilde{K}\big) w,\qquad y\geq 0.
\end{align}
Here $\widetilde{K}\in \mathbb{R}^{(r-n_0) \times (r-n_0-n_{10})}$ is the orthogonal complement of $K$, $X\in\mathbb{R}^{(r-n_0)\times (r-n_0)}$ and $N\in \mathbb{R}^{(n-r)\times (r-n_0-n_{10})}$ are defined by
\begin{align*}
&X= \widehat{A}_{22}-\widehat{A}_{12}^T P_1\Lambda_{1}^{-1}P_1^T
	             \widehat{A}_{12} \\[2mm]
&N =
    -P_1\Lambda_1^{-1}P_{1}^{T}\widehat{A}_{12}\widetilde{K} + P_0(K^TK)^{-1}\Big(\big(K^T\widehat{S}\widetilde{K}\big)\big(\widetilde{K}^T\widehat{S}\widetilde{K}\big)^{-1}\big(\widetilde{K}^TX\widetilde{K}\big)-K^TX\widetilde{K}\Big).
\end{align*}
Moreover, the coefficient matrix $\big(\widetilde{K}^TX\widetilde{K}\big)^{-1}\big(\widetilde{K}^T \widehat{S} \widetilde{K}\big)$ in \eqref{ODEw} is invertible with $(n_+-n_{1+}-n_{10})$ stable eigenvalues.
\end{itemize}

\end{prop}

This proposition is just the Propositions 3.2 and 3.3 in \cite{ZY3}  with $n$, $r$, $A_1$ and $Q$ replaced by $(n-n_0)$, $(r-n_0)$, $\widehat{A}_1$ and $\widehat{Q}$ respectively. 
Thanks to this proposition, $L_1\begin{pmatrix}
    \mu_0\\
    \nu_0
\end{pmatrix}$ 
can be uniquely determined by solving the ODE system \eqref{ODEw}, defined in the half-space $y\geq 0$, if proper initial conditions are prescribed.
Thus the boundary-layer terms $(\mu_0,\nu_0)$ can be determined by using \eqref{relation0-1} and \eqref{mu0nu0-relation}.

Next we turn to the boundary-layer equation \eqref{P0mu} for the other correction term $(\mu_1,\nu_1)$.
This time-dependent equation is a parabolic system of  
second-order partial differential equations, which can be shown as
\begin{prop}\label{prop44}
    The coefficient matrix $P_0^TA_{12}S^{-1}A_{12}^TP_0$ in \eqref{P0mu} is symmetric negative definite. 
\end{prop}
\begin{proof}
Since $S$ is symmetric negative definite, it suffices to show that $A_{12}^TP_0$ has full-rank. From Proposition \ref{prop2.1} (i) and \eqref{hatA12A22}, we know that 
$$
n_{10}=\text{rank}(K) = \text{rank}\big((R_{02}^{\perp})^TA_{12}^TP_0\big) \leq \text{rank}(A_{12}^TP_0).
$$ 
This completes the proof.
\end{proof}
\noindent Due to this proposition, we know that $n_{10}$ boundary conditions should be given for \eqref{P0mu}.

\subsection{Boundary conditions for the asymptotic expansion}\label{section4.1}

In the previous subsections, we have analyzed the equations for the outer solution and boundary-layer correction terms.
It is pointed out that the hyperbolic system \eqref{3.2} and the parabolic system \eqref{P0mu} need proper boundary conditions while the ODE system \eqref{ODEw} requires initial conditions. Here we turn to these initial and boundary conditions. 

Substituting the asymptotic expansion \eqref{3.1} into the boundary condition in \eqref{Boundary conditions} and matching the coefficient of order $O(1)$, we obtain 
\begin{equation}\label{BC-asymptotic}
B\begin{pmatrix}
        \bar{u} + \mu_1 + \mu_0 \\[2mm]
        \bar{v} + \nu_1 + \nu_0
\end{pmatrix}(0,\hat{x},t) = b(\hat{x},t).    
\end{equation}
Due to $BR_0=0$ assumed in \eqref{2.3}, we have $B=B(R_1L_1+R_0L_0)=BR_1L_1$ and thereby
$$
BR_1L_1 \begin{pmatrix}
        \bar{u} + \mu_1 + \mu_0 \\[2mm]
        \bar{v} + \nu_1 + \nu_0
\end{pmatrix}(0,\hat{x},t) = b(\hat{x},t).
$$
By using \eqref{3.3}, \eqref{eq:blcorrection11},
\eqref{P1mu} and Proposition \ref{prop2.1}, the relation \eqref{BC-asymptotic} can be reduced to
\begin{align*} 
BR_1L_1
 \begin{pmatrix}
        \bar{u} + P_0(P_0^T\mu_1) \\[2mm]
        0
\end{pmatrix}(0,\hat{x},t) + BR_1\begin{pmatrix}
    N\\[2mm]
    \widetilde{K}
\end{pmatrix}w(0,\hat{x},t) = b(\hat{x},t).
\end{align*}
On the other hand, we deduce from $L_1R_1=I_{n-n_0}$ and the expression of $R_1$ in \eqref{R0R1} that
$$
L_1
\begin{pmatrix}
I_{n-r} & 0\\[2mm]
0 & R_{02}^{\perp}
\end{pmatrix} = I_{n-n_0} = \begin{pmatrix}
I_{n-r} & 0\\[2mm]
0 & I_{r-n_0}
\end{pmatrix}
\quad \Rightarrow \quad 
L_1
\begin{pmatrix}
I_{n-r}\\[2mm]
0 
\end{pmatrix} = 
\begin{pmatrix}
I_{n-r}\\[2mm]
0 
\end{pmatrix}.
$$
Moreover, we partition $B=(B_u,B_v)$
as in \eqref{partA} and compute $BR_1 = (B_u,B_vR_{02}^{\perp})$. Thus the relation \eqref{BC-asymptotic} becomes
$$
B_u\bar{u}(0,\hat{x},t) + B_uP_0(P_0^T\mu_1)(0,\hat{x},t) + \left(B_uN + B_v R_{02}^{\perp} \widetilde{K} \right)w(0,\hat{x},t) = b(\hat{x},t).
$$

For the ODE system \eqref{ODEw} to have a bounded solution in $y\geq 0$, the initial data $w(0,\hat{x},t)$
should be given on the stable subspace $\text{span}\{R_2^S\}$ with $R_2^S$ the right-stable matrix of the coefficient matrix 
$\big(\widetilde{K}^TX\widetilde{K}\big)^{-1}\big(\widetilde{K}^T S\widetilde{K}\big)$. Proposition \ref{prop2.1} indicates that $R_2^S$ should be an $(r-n_0-n_{10})\times (n_+-n_{1+}-n_{10})$ matrix. 
Thus we express $w(0,\hat{x},t)=R_2^Sw^S(\hat{x},t)$ with $w^S\in \mathbb{R}^{n_+-n_{1+}-n_{10}}$. Using this, we finally obtain
\begin{align}\label{4.36}
B_u\bar{u}(0,\hat{x},t) + B_uP_0(P_0^T\mu_1)(0,\hat{x},t) + \left(B_uNR_2^S + B_v R_{02}^{\perp} \widetilde{K}R_2^S \right)  w^S(\hat{x},t) = b(\hat{x},t).
\end{align}
From this relation, we will derive the reduced boundary condition for $\bar{u}$ and determine $P_0^T\mu_1(0,\hat{x},t)$ and $w^S(\hat{x},t)$. 

\subsection{Reduced boundary conditions}\label{Section4.4}

To derive the reduced boundary condition from \eqref{4.36}, we follow \cite{Y2,ZY3} and compute the right-stable matrix $R_M^S(\xi,\omega,\eta)$ in the characteristic GKC for sufficiently large $\eta$. Recall that both the characteristic GKC and the matrix defined in \eqref{M}:
$$
M(\xi,\omega,\eta) = \widehat{A}_1^{-1}\Big[G_{11}(\xi,\omega,\eta)-G_{10}(\xi,\omega,\eta)G_{00}^{-1}(\xi,\omega,\eta)G_{01}(\xi,\omega,\eta)\Big]
$$
 involve the matrix $R_1$.
Firstly, we have
\begin{lemma}\label{lemma4.1}
    With the elaborately selected matrix $R_1$ in \eqref{R0R1}, the matrix $M(\xi,\omega,\eta)$ has the following form 
    \begin{align}\label{M-block}
    M(\xi,\omega,\eta) 
    =&~\begin{pmatrix}
        A_{11} & \widehat{A}_{12} \\[2mm]
        \widehat{A}_{12}^T & \widehat{A}_{22} 
    \end{pmatrix}^{-1}
    \left[ \eta \begin{pmatrix}
        0 & 0 \\[2mm]
        0 & \widehat{S} 
    \end{pmatrix} - 
    \begin{pmatrix}
        \xi I_{n-r} + C(\omega) & \star \\[2mm]
        \star & \star
    \end{pmatrix}
    \right]+O\Big(\frac{1}{\eta}\Big)
\end{align}
for sufficiently large $\eta$. 
Here 
$$
C(\omega)=i\sum_{j=2}^d \omega_jA_{j11},
$$ 
the notation $\star$
means a matrix independent of the parameter $\eta$ and the details are omitted.
\end{lemma}

\begin{proof}
Recall the definition in \eqref{defM} that
\begin{equation}\label{H}
G_{kl}(\xi,\omega,\eta) = \eta (R_k^T Q R_l) + R_k^THR_l \quad \text{with} \quad H = - \xi I - i\sum_{j=2}^d \omega_jA_j.  
\end{equation}
For sufficiently large $\eta$, it is not difficult to compute the inverse matrix
$$
G_{00}^{-1}(\xi,\omega,\eta) = \frac{1}{\eta} (R_0^TQR_0)^{-1} - \frac{1}{\eta^2} (R_0^TQR_0)^{-1}(R_0^THR_0)(R_0^TQR_0)^{-1} + O\Big(\frac{1}{\eta^3}\Big).
$$    
Thus we have 
\begin{align*}
&G_{11}(\xi,\omega,\eta)-G_{10}(\xi,\omega,\eta)G_{00}^{-1}(\xi,\omega,\eta)G_{01}(\xi,\omega,\eta) \\[3mm]
=&~ \eta\Big(R_1^TQR_1 + \frac{1}{\eta} R_1^THR_1\Big) - \eta\Big(R_1^TQR_0 + \frac{1}{\eta} R_1^THR_0\Big)\\[2mm]
&~\Big((R_0^TQR_0)^{-1} - \frac{1}{\eta} (R_0^TQR_0)^{-1}(R_0^THR_0)(R_0^TQR_0)^{-1}\Big)
\Big(R_0^TQR_1 + \frac{1}{\eta} R_0^THR_1\Big)+O\Big(\frac{1}{\eta}\Big) \\[2mm]
=&~\eta \widehat{Q} + \widehat{H} +  O\Big(\frac{1}{\eta}\Big)
\end{align*}
with coefficients 
\begin{align*}
\widehat{Q} =&~ R_1^TQR_1 - R_1^TQR_0 (R_0^TQR_0)^{-1} R_0^TQR_1,\\[3mm]
\widehat{H} =&~ R_1^THR_1 - (R_1^TQR_0)(R_0^TQR_0)^{-1}(R_0^THR_1)
- (R_1^THR_0)(R_0^TQR_0)^{-1}(R_0^TQR_1)\quad \\[2mm]
&~+ (R_1^TQR_0)(R_0^TQR_0)^{-1}(R_0^THR_0)(R_0^TQR_0)^{-1}(R_0^TQR_1). 
\end{align*}
Notice that the matrix $\widehat{Q}$ has been computed in \eqref{new:hatQ} and $\widehat{H}$ can be reorganized as 
\begin{align*}
\widehat{H} =&~ \Big[R_1^T - (R_1^TQR_0)(R_0^TQR_0)^{-1}R_0^T\Big] H  \Big[ R_1 - R_0(R_0^TQR_0)^{-1}(R_0^TQR_1)\Big].
\end{align*}
By the definition of $H$ in \eqref{H}, we can write 
\begin{equation*}
H = 
\begin{pmatrix}
    - \xi I_{n-r} - C(\omega) & \star \\[2mm]
    \star & \star
\end{pmatrix}\quad \text{with} \quad C(\omega)=i\sum_{j=2}^d \omega_jA_{j11}. 
\end{equation*}
Moreover, we use the expression of $R_1$ in 
\eqref{R0R1} to compute
\begin{align*}
R_1 - R_0(R_0^TQR_0)^{-1}R_0^TQR_1 = 
&\begin{pmatrix}
I_{n-r} & 0 \\[2mm]
0 & R_{02}^{\perp}
\end{pmatrix}
-
\begin{pmatrix}
R_{01} \\[2mm]
R_{02}
\end{pmatrix}\left(R_{02}^TSR_{02}\right)^{-1}
\begin{pmatrix}
R_{01}^T & R_{02}^T
\end{pmatrix}
\begin{pmatrix}
0 & 0 \\[2mm]
0 & SR_{02}^{\perp}
\end{pmatrix}
\\[2mm]
=&\begin{pmatrix}
I_{n-r} & \star \\[2mm]
0 & \star
\end{pmatrix}.
\end{align*}
Thus it follows that 
\begin{align*}
\widehat{H} =&~ \begin{pmatrix}
I_{n-r} & 0 \\[2mm]
 \star & \star
\end{pmatrix}
\begin{pmatrix}
    - \xi I_{n-r} - C(\omega) & \star \\[2mm]
    \star & \star
\end{pmatrix}
\begin{pmatrix}
I_{n-r} & \star \\[2mm]
 0 & \star
\end{pmatrix}
=
-\begin{pmatrix}
 \xi I_{n-r} + C(\omega) & \star \\[2mm]
 \star & \star
\end{pmatrix}.
\end{align*}
This completes the proof.
\end{proof}

Observe that, without the high-order term $O(1/\eta)$, 
the expression of $M(\xi,\omega,\eta)$ in \eqref{M-block} is the same as that in \cite{ZY3} by dropping out the symbol "$~\widehat{~}~$". Thus we can refer to the argument in Section 4 of \cite{ZY3}, with a series of subtle matrix transformations, to have the following lemma. Note that the high-order term does not affect the argument in \cite{ZY3}.
\begin{lemma}[\cite{ZY3}]\label{lemma4.2}
As $\eta$ goes to infinity, the right-stable matrix for $M(\xi,\omega,\eta)$ can be expressed by
\begin{align*}
R_M^S(\xi,\omega,\infty) = 
\begin{pmatrix}
\Big[P_1-P_0\left[\xi I_{n_{10}}+P_0^TC(\omega)P_0\right]^{-1}(P_0^TC(\omega)P_1)\Big]R_{1}^S  & P_0 & NR_{2}^S \\[2mm]
0 & 0  & \widetilde{K} R_{2}^S
\end{pmatrix}.
\end{align*}
Here the matrices $\widetilde{K}$ and $N$ are given in Proposition \ref{prop2.1}.
Moreover, $R_1^S\in\mathbb{R}^{(n-r)\times n_{1+}}$ is the right-stable matrix for 
$$
M_1(\xi,\omega) = -\Lambda_1^{-1}\bigg(\big[\xi I +P_1^TC(\omega)P_1\big]-P_1^TC(\omega)P_0\big[\xi I + P_0^TC(\omega)P_0 \big]^{-1}P_0^TC(\omega)P_1 \bigg)
$$
and $R_2^S\in \mathbb{R}^{(r-n_0-n_{10})\times (n_+-n_{1+}-n_{10})}$ is the right-stable matrix for 
$$
M_2=(\widetilde{K}^TX\widetilde{K})^{-1}(\widetilde{K}^T \widehat{S}\widetilde{K}).
$$
\end{lemma}

\begin{remark}
$M_2$ is just the coefficient matrix for the ODE system  \eqref{ODEw}. The matrix $M_1(\xi,\omega)$ is related to the Uniform Kreiss Condition(UKC) for the equilibrium system \eqref{3.2}.
To see this, we consider the corresponding eigenvalue problem of \eqref{3.2}:
$$
\xi\widehat{u}+A_{11}\widehat{u}_{x_1}+C(\omega)\widehat{u} = 0 
\quad \text{with} \quad C(\omega)=i\sum_{j=2}^d \omega_jA_{j11}.
$$
As in \eqref{eq:3.4} and \eqref{equationM}, we can easily deduce that 
\begin{align*}
P_0^T\widehat{u}=&~-[\xi I+P_0^TC(\omega)P_0]^{-1}P_0^TC(\omega)P_1(P_1^T\widehat{u}), \\[2mm]
(P_1^T\widehat{u})_{x_1}=&~-\Lambda_1^{-1}\left(\big[\xi I +P_1^TC(\omega)P_1\big] -P_1^TC(\omega)P_0\big[\xi I + P_0^TC(\omega)P_0 \big]^{-1}P_0^TC(\omega)P_1 \right)(P_1^T\widehat{u})\\[2mm]
\equiv &~M_1(\xi,\omega)(P_1^T\widehat{u}).
\end{align*}
By using the same argument as that in Section \ref{sec3}, the UKC for the boundary condition $\bar{B}\bar{u}(0,\hat{x},t)=\bar{b}(\hat{x},t)$ of the equilibrium system reads as 
$$
|\det\{\bar{B}P_1R_{1}^S\}|\geq c_K\sqrt{\det\{R_{1}^{S*}R_{1}^S \}}
$$ 
with $R_1^S$ the right-stable matrix for 
$M_1(\xi,\omega)$. 
\end{remark}

Now we state our main result with the partition $B=(B_u,B_v)$ as in Subsection \ref{section4.1}.
\begin{theorem}\label{theorem43}
Under the assumptions in Section \ref{Section2}, there exists a full-rank matrix $B_o\in\mathbb{R}^{n_{1+}\times n_+}$
such that the relation
\begin{equation}\label{reduced-BC}
B_oB_u \bar{u}(0, \hat{x},t) = B_ob(\hat{x},t),    
\end{equation}
as a boundary condition for the equilibrium system \eqref{3.2}, satisfies $B_oB_uP_0 = 0$ and the Uniform Kreiss
Condition \cite{MO} for the characteristic IBVPs:
$$
    |\text{det}\{B_0B_uP_1R_1^S\}|\geq \bar{c}_K\sqrt{\text{det}\{R_1^{S*}R_1^S\} }
$$
with $\bar{c}_K$ a positive constant and $R_1^S$ the right-stable matrix of $M_1(\xi,\omega)$.
\end{theorem}

\begin{proof}
The proof is quite similar to that in \cite{ZY3}. For completeness, we present the full details here. Taking $\eta\rightarrow \infty$ in the characteristic GKC, we have 
$$
|\text{det}\{BR_1 R_M^S(\xi,\omega,\infty)\}|\geq c_K\sqrt{\text{det}\{R_M^{S*}(\xi,\omega,\infty)R_M^S(\xi,\omega,\infty)\} }.
$$
This indicates the invertibility of $BR_1 R_M^S(\xi,\omega,\infty)$. 
By using the expression of $R_M^S(\xi,\omega,\infty)$ in Lemma \ref{lemma4.2}, we can write
$$
BR_1 R_M^S(\xi,\omega,\infty) = (Y_1,Y_2,Y_3)
$$
with
\begin{align*}
    &Y_1 = B_u \Big[P_1-P_0\left[\xi I_{n_{10}}+P_0^TC(\omega)P_0\right]^{-1}(P_0^TC(\omega)P_1)\Big]R_{1}^S, \\[2mm]
&Y_2 = B_u P_0,\qquad 
Y_3 = B_uNR_{2}^S  + B_v R_{02}^{\perp} \widetilde{K}R_{2}^S
.
\end{align*}
Notice that $Y_1\in \mathbb{R}^{n_+\times n_{1+}}$, $Y_2\in \mathbb{R}^{n_+\times n_{10}}$ and $Y_3 \in \mathbb{R}^{n_+\times (n_+-n_{1+}-n_{10})}$. 
Then there is a full-rank matrix $B_o\in \mathbb{R}^{n_{1+} \times n_{+}}$ such that 
$$
B_o\begin{pmatrix}
    Y_2 & Y_3
\end{pmatrix} = 0.
$$
Taking $\widetilde{B}_o\in \mathbb{R}^{ (n_+-n_{1+})\times n_+}$ such that $\begin{pmatrix}
    B_o\\
    \widetilde{B}_o
\end{pmatrix}$ is invertible, it follows that
\begin{align*}
\begin{pmatrix}
    B_o\\[2mm]
    \widetilde{B}_o
\end{pmatrix}BR_1 R_M^S(\xi,\omega,\infty)=
\begin{pmatrix}
    B_o\\[2mm]
    \widetilde{B}_o
\end{pmatrix}
\begin{pmatrix}
    Y_1 & Y_2 & Y_3
\end{pmatrix}
=
\begin{pmatrix}
    B_oY_1 & 0 & 0 \\[2mm]
    \widetilde{B}_oY_1 & \widetilde{B}_oY_2 & \widetilde{B}_oY_3
\end{pmatrix}.
\end{align*}

By Lemma \ref{lemma4.2}, $R_2^S$ (in $Y_3$) is a right-stable matrix of $M_2=(\widetilde{K}^TX\widetilde{K})^{-1}\widetilde{K}^T\widehat{S}\widetilde{K}$, which is independent of parameters $\xi$ and $\omega$. Thus $B_o$ and $\widetilde{B}_o$ are independent of the parameters and we have
\begin{align*}
    |\text{det}\{B_oY_1\}|=&~\left|\text{det}\{(\widetilde{B}_oY_2,~\widetilde{B}_oY_3)\}^{-1}\right|
    \left|\text{det}\left\{\begin{pmatrix}
    B_o\\[2mm]
    \widetilde{B}_o
\end{pmatrix}\right\}\right|
\left|\text{det}\{BR_1 R_M^S(\xi,\omega,\infty)\}\right|\\[2mm]
\geq &~ c_1 \left|\text{det}\{BR_1 R_M^S(\xi,\omega,\infty)\}\right|
\end{align*}
with $c_1>0$ a constant independent of the parameters. 
Since $B_oY_2=B_oB_uP_0=0$, it follows from the expression of $Y_1$ that $B_oY_1=B_oB_uP_1R_{1}^S$.
Thus the last inequality becomes 
\begin{equation}\label{proof4.3-eq1}
|\text{det}\{B_oB_uP_1R_{1}^S\}| \geq c_1 \left|\text{det}\{BR_1 R_M^S(\xi,\omega,\infty)\}\right|.
\end{equation}

Next we use Lemma \ref{lemma4.2} and rewrite
\begin{eqnarray*}
R_M^S(\xi,\omega,\infty) = \begin{pmatrix}
P_1R_1^S & P_0 & NR_{2}^S \\[2mm]
0 & 0 & \widetilde{K} R_{2}^S
\end{pmatrix}
\begin{pmatrix}
I & 0 & 0 \\[1mm]
X_1 & I & 0 \\[1mm]
0 & 0 & I
\end{pmatrix},
\end{eqnarray*}
where $X_1 = -[\xi I_{n_{10}}+P_0^TC(\omega)P_0]^{-1}(P_0^TC(\omega)P_1) R_{1}^S$. Moreover, it follows from the expression of $N$ in Proposition \ref{prop2.1}:
$$
N = X_2\widetilde{K}+P_0X_3, 
$$ 
with $X_2$ and $X_3$ clearly defined and independent of $(\xi,\omega)$, that 
\begin{eqnarray*}
\begin{pmatrix}
P_1R_1^S & P_0 & NR_{2}^S \\[2mm]
0 & 0 & \widetilde{K} R_{2}^S
\end{pmatrix}
= \begin{pmatrix}
I & X_2 \\[2mm]
0 & I
\end{pmatrix}
\begin{pmatrix}
P_1R_1^S & P_0 & 0\\[2mm]
0 & 0 & \widetilde{K} R_{2}^S
\end{pmatrix}
\begin{pmatrix}
~I & ~0 & 0 \\[2mm]
~0 & ~I & X_3R_2^S \\[2mm]
~0 & ~0 & I
\end{pmatrix}.
\end{eqnarray*}
Thus we deduce from Lemma 3.3 in \cite{Y2} that
\begin{align*}
\sqrt{\text{det}\{R_M^{S*}(\xi,\omega,\infty)R_M^S(\xi,\omega,\infty)\}}
\geq &~ c_2
\sqrt{\text{det}\left\{
\begin{pmatrix}
R_1^{S*}P_1^T \\[1mm]
P_0^T 
\end{pmatrix}
\begin{pmatrix}
P_1R_1^S & P_0 
\end{pmatrix}
\right\}
\text{det}\left\{
R_{2}^{S*} \widetilde{K}^T\widetilde{K} R_{2}^S
\right\}} \\[2mm]
= &~c_2
\sqrt{\text{det}\{R_1^{S*}R_1^S\} \text{det}\left\{
R_{2}^{S*} \widetilde{K}^T\widetilde{K} R_{2}^S
\right\} }
\end{align*}
with $c_2>0$ depending only on $X_2$. Here we have used the orthogonality of $P_0$ and $P_1$. 
Note that $R_{2}^{S*} \widetilde{K}^T\widetilde{K} R_{2}^S$ is positive definite and independent of the parameters $(\xi,\omega)$. We conclude from the last inequality, together with \eqref{proof4.3-eq1} and the characteristic GKC, that there is a constant $\bar{c}_K$ such that
$$
    |\text{det}\{B_0B_uP_1R_1^S\}|\geq \bar{c}_K\sqrt{\text{det}\{R_1^{S*}R_1^S\} }.
$$
This completes the proof.
\end{proof}

Thanks to Theorem \ref{theorem43}, the outer solution $\bar{u}$ can be uniquely determined by solving the equilibrium system \eqref{3.2} with the reduced boundary condition \eqref{reduced-BC} and proper initial data. Furthermore, we can use \eqref{4.36} to 
obtain the desired boundary conditions for the parabolic system \eqref{P0mu} and initial conditions for the ODE system \eqref{ODEw}.
Indeed, multiplying $\widetilde{B}_o$ on the left of \eqref{4.36} yields
\begin{equation}\label{eq-otherBC}
\widetilde{B}_o\begin{pmatrix}
    B_uP_0, & B_uNR_{2}^S  + B_v R_{02}^{\perp} \widetilde{K}R_{2}^S
\end{pmatrix}
\begin{pmatrix}
P_0^T\mu_1 \\[2mm]
w^S
\end{pmatrix} 
= 
\widetilde{B}_o b(\hat{x},t) - \widetilde{B}_o B_u \bar{u}(0,\hat{x},t).
\end{equation}
Once $\bar{u}$ is solved, this is a system of $(n_+-n_{1+})$ linear algebraic equations for the $n_{10}$ variables $P_0^T\mu_1$ and $(n_+-n_{1+}-n_{10})$ variables $w^S$. 
Following the proof of Theorem \ref{theorem43}, we notice that the coefficient matrix is just $\widetilde{B}_o (Y_2,Y_3)$, which is invertible. Thus $P_0^T\mu_1$ and $w^S$ can be uniquely determined from \eqref{eq-otherBC}, giving the desired boundary and initial conditions.

\section{Validity}\label{section5}

In this section, we show the validity of the reduced boundary condition \eqref{reduced-BC} by examining the discrepancy between the exact solution $U^{\epsilon}$ to the IBVP \eqref{new-2.1} with \eqref{Boundary conditions} and the solution to the 
equilibrium system with \eqref{reduced-BC}. For this purpose, we will estimate the $L^2$-error between $U^{\epsilon}$ and the asymptotic solution $U_\epsilon$ constructed in Section \ref{sec3} when $\epsilon$ is small.
 
\subsection{Asymptotic solutions}
In order to focus on the boundary-layer behaviours, we choose the initial data such that the initial-layer can be neglected. 
To this end, the initial value $U_0(x_1,\hat{x})$ for the relaxation system \eqref{new-2.1} is taken to be in equilibrium:
\begin{equation}\label{5.1}
	U_0(x_1,\hat{x})=\left({\begin{array}{*{20}c}
  \vspace{1.5mm}u_0(x_1,\hat{x}) \\
                0
\end{array}}\right),
\end{equation}
where $u_0(x_1,\hat{x})$ represents the first $(n-r)$ components of $U_0(x_1,\hat{x})$.  
Moreover, we assume that the initial value $U_0(x_1,\hat{x})$ and the boundary condition in \eqref{Boundary conditions} are compatible:
\begin{equation}\label{5.2}
  BU_0(0,\hat{x})=b(\hat{x},0),\qquad \text{for}~\hat{x}\in \mathbb{R}^{d-1}.
\end{equation} 
From \eqref{5.1} and \eqref{5.2}, we see that 
$B_oB_uu_0(0,\hat{x})=B_ob(\hat{x},0)$,
meaning that the reduced boundary condition \eqref{reduced-BC} for the equilibrium system is compatible with the initial value $u_0(x_1,\hat{x})$.

About the asymptotic solution $U_{\epsilon}$ in \eqref{3.1}, we recall from \eqref{3.3}, \eqref{eq:blcorrection11}, \eqref{P1mu} that the coefficients satisfy
$$
\bar{v}=0, \quad \nu_1=0, \quad P_1^T\mu_1=0.
$$
The terms $\bar{u}$, $P_0^T\mu_1$ and $(\mu_0,\nu_0)$ solve the differential equations \eqref{3.2}, \eqref{P0mu} and \eqref{eq:blcorrection12} respectively. According to Theorem \ref{theorem43} and the discussion at the end of Section \ref{section4}, proper boundary and initial conditions have been determined for these equations. 
The initial data for the PDEs \eqref{3.2} and \eqref{P0mu} are the same as those in our previous work \cite{ZY3}:
$$
\bar{u}(x_1,\hat{x},0) = u_0(x_1,\hat{x}),\qquad 
P_1^T\mu_1(z,\hat{x},0)\equiv 0,
$$
which imply
\begin{equation}\label{new:5.4}
U_\epsilon(x_1,\hat{x},0) = U_0(x_1,\hat{x}) = U^\epsilon(x_1,\hat{x},0).
\end{equation}
At last, we determine $\mu_2$ and $\nu_2$ from the relations \eqref{eq:blcorrection16} and
\eqref{eq:blcorrection14}. Note that $\mu_2$ is not fully determined by the equation \eqref{eq:blcorrection16}. 
The interested reader is referred to Section 5.1 in \cite{ZY3} for further details.
For the coefficients thus determined, we make the following regularity assumption.
\begin{assumption}\label{asp5.1}
The expansion coefficients in \eqref{3.1} satisfy
\begin{eqnarray*} 
\left\{{\begin{array}{*{20}l}
\bar{u}, ~P_0^T\mu_1\in L^2([0,T]\times H^1(\mathbb{R}^+\times \mathbb{R}^{d-1})),\\[3mm]
(\mu_0,\nu_0), ~(\mu_2, \nu_2)\in H^1([0,T]\times \mathbb{R}^+ \times \mathbb{R}^{d-1}).
\end{array}}\right. 
\end{eqnarray*}
\end{assumption}

\subsection{Error estimate}

Denote the difference between the exact solution $U^{\epsilon}$ and the asymptotic solution $U_\epsilon$ by
$$
W:=U^\epsilon-U_\epsilon.
$$
We firstly derive the equation for $W$. To this end, we recall the linear operator
$$
\mathcal{L}(U):=\partial_{t}U +\sum\limits_{j=1}^d A_j\partial_{x_j}U-QU/\epsilon 
$$
and compute $\mathcal{L}(W)=\mathcal{L}(U^{\epsilon})-\mathcal{L}(U_{\epsilon})$. Since $U^{\epsilon}$ satisfies the equation \eqref{new-2.1}, we know that $\mathcal{L}(U^{\epsilon})=0$. On the other hand, from \eqref{Residual} and \eqref{residual-1} we know that
$$
\mathcal{L}(U_\epsilon)=\mathcal{L}(U_{outer})+\mathcal{L}(U_{layer}),\qquad 
\mathcal{L}(U_{outer}) = 
\begin{pmatrix}
    0 \\
  \sum\limits_{j=1}^d A_{j12}^T\partial_{x_j}\bar{u}
\end{pmatrix}.
$$
Furthermore, we use \eqref{eq:residualbl} to compute $\mathcal{L}(U_{layer})$ in \eqref{residual-2} as
\begin{align*} 
\mathcal{L}(U_{layer})=&\begin{pmatrix}
\partial_t \mu_0+ \sum\limits_{j=2}^d\Big[A_{j11}\partial_{x_j}\mu_0 + A_{j12}\partial_{x_j}\nu_0\Big] \\
  \partial_t\nu_{0}+ \sum\limits_{j=2}^d\Big[A_{j12}^T\partial_{x_j}(\mu_0+\mu_1)+A_{j22}\partial_{x_j}\nu_{0}\Big]+A_{12}^T\partial_z\mu_2+A_{22}\partial_z\nu_2
\end{pmatrix}\nonumber \\[2mm]
&+ \sqrt{\epsilon}
\left[
\partial_t
\begin{pmatrix}
\mu_2 \\[2mm]
\nu_2    
\end{pmatrix}
+\sum\limits_{j=2}^d
\begin{pmatrix}
A_{j11} & A_{j12}\\[2mm]
A_{j12}^T & A_{j22}    
\end{pmatrix}
\partial_{x_j}
\begin{pmatrix}
\mu_2 \\[2mm]
\nu_2    
\end{pmatrix}
\right].
\end{align*}
In summary, we can write 
$$
\mathcal{L}(U_\epsilon)
=\mathcal{L}(U_{outer})+\mathcal{L}(U_{layer})
=\begin{pmatrix}
    E_1\\[1mm]
    E_2
\end{pmatrix}
$$
with $E_1$ and $E_2$ clearly defined. 
By the regularity Assumption \ref{asp5.1}, we see that $E_2\in L^2([0,T]\times \mathbb{R}^+\times \mathbb{R}^{d-1})$. 
Moreover, since 
$$
\int_{[0,T]\times R^+\times R^{d-1}}\left|(\mu_0,\nu_0)\left(\frac{x_1}{\epsilon},\hat{x},t\right)\right|^2dtdx_1d\hat{x}=
\epsilon\int_{[0,T]\times R^+\times R^{d-1}}\left|(\mu_0,\nu_0)(y,\hat{x},t)\right|^2dtdyd\hat{x},
$$
we have $\|E_1\|_{L^2([0,T]\times R^+\times R^{d-1})}\leq C\sqrt{\epsilon}$.
Consequently, the equation for $W$ is 
\begin{align}\label{tem5.5}
\partial_tW+
\sum\limits_{j=1}^d A_j
\partial_{x_j} W=\frac{1}{\epsilon}
\begin{pmatrix}
0 & 0 \\[2mm]
0 & S    
\end{pmatrix}W 
-
\begin{pmatrix}
E_{1} \\[2mm]
E_{2}   
\end{pmatrix}.
\end{align}

Next we turn to the initial and boundary conditions for $W$. 
Clearly, it follows from \eqref{new:5.4} that 
\begin{align}\label{5.9}
W(x_1,\hat{x},0) \equiv 0.
\end{align}
From the boundary conditions in \eqref{Boundary conditions} and \eqref{BC-asymptotic}, 
it is easy to see that the boundary value of $W$ satisfies 
\begin{eqnarray}\label{tem5.6}
  BW(0,\hat{x},t)=\sqrt{\epsilon}
  \begin{pmatrix}
  \mu_2 \\[2mm]
  \nu_2    
  \end{pmatrix}(0,\hat{x},t)\equiv g(\hat{x},t).
\end{eqnarray} 
By Assumption \ref{asp5.1}, we have $\mu_2(0,\hat{x},t),\nu_2(0,\hat{x},t)\in L^2([0,T]\times \mathbb{R}^{d-1})$ and thereby
\begin{equation*} 
	\|g\|_{L^2([0,T]\times \mathbb{R}^{d-1})}\leq  C\sqrt{\epsilon}.
\end{equation*} 
Combining \eqref{tem5.5}-\eqref{tem5.6}, we have the following IBVP for $W$:
\begin{align*} 
\left\{
{\begin{array}{*{20}l}
  \vspace{1.5mm}\partial_tW+
\sum\limits_{j=1}^d A_j
\partial_{x_j} W=\dfrac{1}{\epsilon}
\begin{pmatrix}
0 & 0 \\[2mm]
0 & S    
\end{pmatrix}W 
-
\begin{pmatrix}
E_{1} \\[2mm]
E_{2}   
\end{pmatrix}, \\[3mm]
BW(0,\hat{x},t)=g(\hat{x},t),\\[2mm]
W(x_1,\hat{x},0)=0.
\end{array}}
\right.
\end{align*}

To estimate $W$, we follow \cite{GKO} and decompose $W=W_1+W_2$. Here $W_1$ satisfies
\begin{align}\label{5.10}
\left\{
{\begin{array}{*{20}l}
  \vspace{1.5mm}\partial_tW_1+
\sum\limits_{j=1}^d A_j
\partial_{x_j} W_1=\dfrac{1}{\epsilon}
\begin{pmatrix}
0 & 0 \\[2mm]
0 & S    
\end{pmatrix}W_1 
-
\begin{pmatrix}
E_{1} \\[2mm]
E_{2}   
\end{pmatrix}, \\[4mm]
R_+^TW_1(0,\hat{x},t)=0,\\[2mm]
W_1(x_1,\hat{x},0)=0,
\end{array}}
\right.
\end{align}
and $W_2$ satisfies
\begin{align}\label{5.11}
\left\{
{\begin{array}{*{20}l}
  \vspace{1.5mm}\partial_tW_2+
\sum\limits_{j=1}^d A_j
\partial_{x_j} W_2=\dfrac{1}{\epsilon}
\begin{pmatrix}
0 & 0 \\[2mm]
0 & S    
\end{pmatrix}W_2, \\[2mm]
BW_2(0,\hat{x},t)=g(\hat{x},t)-BW_1(0,\hat{x},t),\\[2mm]
W_2(x_1,\hat{x},0)=0.
\end{array}}\qquad \qquad 
\right.
\end{align}
Note that the boundary condition for $W_1$ is constructed artificially and the matrix $R_+$ consists of eigenvectors of $A_1$ associated with the positive eigenvalues.
Since $A_1$ is symmetric, 
we can take $R_+$ such that 
$(R_+,R_-,R_0)$ is an orthonormal matrix satisfying
\begin{equation*}
\begin{pmatrix}
    R_+^T \\[1mm]
    R_-^T \\[1mm]
    R_0^T
\end{pmatrix}
A_1
\begin{pmatrix}
    R_+ & R_- & R_0
\end{pmatrix}=
\begin{pmatrix}
    \Lambda_+ & & \\[1mm]
    & \Lambda_- & \\[1mm]
    & & 0
\end{pmatrix}.
\end{equation*}
Here $\Lambda_+$ and $\Lambda_-$ are diagonal matrices whose entries are $n_+$ positive eigenvalues and $(n-n_0-n_+)$ negative eigenvalues of $A_1$. 
Based on this decomposition, we take $R_1=(R_+,R_-)$. Recall Remark  \ref{remark3.3} that the characteristic GKC does not rely on the choice of $R_0$ and $R_1$.

For IBVPs \eqref{5.10} and \eqref{5.11}, we have the following conclusions.

\begin{lemma}\label{lemma5.1}
The IBVP \eqref{5.10} has an unique solution $W_1=W_1(x_1,\hat{x},t)$ satisfying
\begin{align*}
&\max_{t\in[0,T]}\|W_1(\cdot,\cdot,t)\|_{L^2(\mathbb{R}^+\times \mathbb{R}^{d-1})}^2+ \|R_1^TW_1|_{x_1=0}\|_{L^2([0,T]\times \mathbb{R}^{d-1})}^2 \nonumber \\[2mm]
\leq &~ C(T)\bigg(\|E_{1}\|^2_{L^2([0,T]\times \mathbb{R}^+\times \mathbb{R}^{d-1})}
+ \epsilon\|E_{2}\|^2_{L^2([0,T]\times \mathbb{R}^+\times \mathbb{R}^{d-1})}\bigg). 
\end{align*}
Here $C(T)$ is a generic constant depending only on $T$.
\end{lemma}

\begin{lemma}\label{lemma5.2}
The IBVP \eqref{5.11} has an unique solution $W_2=W_2(x_1,\hat{x},t)$ satisfying 
\begin{align*}
&\max_{t\in[0,T]}\|W_2(\cdot,\cdot,t)\|_{L^2(\mathbb{R}^+\times \mathbb{R}^{d-1})}^2+ \|R_1^TW_2|_{x_1=0}\|_{L^2([0,T]\times \mathbb{R}^{d-1})}^2 \nonumber \\[2mm]
\leq &~ C(T)\bigg(\|R_1^TW_1|_{x_1=0}\|_{L^2([0,T]\times \mathbb{R}^{d-1})}^2+\|g\|_{L^2([0,T]\times \mathbb{R}^{d-1})}^2\bigg).
\end{align*}
\end{lemma}
\noindent Recall that $\|E_2\|_{L^2([0,T]\times \mathbb{R}^+\times \mathbb{R}^{d-1})}\leq C$, $\|E_1\|_{L^2([0,T]\times \mathbb{R}^+\times \mathbb{R}^{d-1})}\leq C \sqrt{\epsilon}$, and $\|g\|_{L^2([0,T]\times \mathbb{R}^{d-1})}\leq  C\sqrt{\epsilon}$. 
The above two lemmas immediately give

\begin{theorem}\label{thm5.3}
Under the characteristic GKC, the basic assumptions in Section \ref{Section2}, the regularity Assumption \ref{asp5.1}, and the constant hyperbolicity assumption \cite{BS}, there exists a constant $C(T)>0$ such that the error estimate 
\begin{align*}
\|(U^\epsilon-U_\epsilon)(\cdot,\cdot,t)\|_{L^2(\mathbb{R}^+\times \mathbb{R}^{d-1})}\leq C(T) \epsilon^{1/2} 
\end{align*}
holds for all time $t\in[0,T]$.
\end{theorem}
\noindent From this theorem and the expansion of $U_\epsilon$ in \eqref{3.1}, we can immediately see the $L^2$-convergence of $U^\epsilon$ to the solution $(\bar{u},\bar{v})$ of the equilibrium system with the reduced boundary condition \eqref{reduced-BC} as $\epsilon$ goes to zero.

It remains to prove Lemma \ref{lemma5.1} and Lemma \ref{lemma5.2}, which is similar to those in \cite{ZY3}. For completeness, we present the details here. 

\begin{proof}[Proof of Lemma \ref{lemma5.1}]
Clearly, the boundary condition $R_+^TW_1(0,\hat{x},t)=0$ satisfies the Uniform Kreiss Condition. Thus it follows from the existence theory in \cite{BS} that there exists a unique solution $W_1\in C([0,T];L^2(\mathbb{R}^+\times \mathbb{R}^{d-1}))$.

For the estimate, we multiply \eqref{5.10} with $W_1^T$ from the left to obtain
$$
\frac{d}{dt}(W_1^TW_1)+\sum_{j=1}^d(W_1^TA_jW_1)_{x_j}= \dfrac{2}{\epsilon}
W_1^T
\begin{pmatrix}
0 & 0 \\[2mm]
0 & S    
\end{pmatrix}
W_1-2(W_1^{I})^TE_1-2(W_1^{II})^TE_{2}.
$$
Here $W_1^I$ represents the first $(n-r)$ components of $W_1$ and $W_1^{II}$ represents the other $r$ components. From the last equation, we use the negative-definiteness of $S$ and derive
\begin{align*}
\frac{d}{dt}(W_1^TW_1)+\sum_{j=1}^d(W_1^TA_jW_1)_{x_j}\leq & -\frac{c_0}{\epsilon}|W_1^{II}|^2+ 2|(W_1^{II})^TE_{2}| + 2|(W_1^{I})^TE_{1}| \\[2mm]
\leq &  -\frac{c_0}{2\epsilon}|W_1^{II}|^2 + \frac{2 \epsilon}{c_0}|E_{2}|^2 +|E_1|^2 + |W_1|^2 
\end{align*}
with $c_0 > 0$ a constant. Integrating the last inequality over $(x_1,\hat{x})\in[0,+\infty)\times \mathbb{R}^{d-1}$ yields
\begin{align}
&~\frac{d}{dt}\|W_1(\cdot,\cdot,t)\|_{L^2(\mathbb{R}^+\times \mathbb{R}^{d-1})}^2 - \int_{\mathbb{R}^{d-1}}W_1^T(0,\hat{x},t) A_1W_1(0,\hat{x},t) d\hat{x}\nonumber \\[2mm]
 \leq &~\frac{2 \epsilon}{c_0}\|E_{2}\|_{L^2(\mathbb{R}^+\times \mathbb{R}^{d-1})}^2 +  \|E_1\|_{L^2(\mathbb{R}^+\times \mathbb{R}^{d-1})}^2 + \|W_1\|_{L^2(\mathbb{R}^+\times \mathbb{R}^{d-1})}^2  .\label{5.14}
\end{align}
Since $A_1=R_+\Lambda_+R_+^T + R_-\Lambda_-R_-^T$ and $R_+^TW_1(0,\hat{x},t)=0$, it follows that 
\begin{align}
-W_1^T(0,\hat{x},t)A_1W_1(0,\hat{x},t)=& -W_1^T(0,\hat{x},t)(R_+\Lambda_+R_+^T + R_-\Lambda_-R_-^T)W_1(0,\hat{x},t)\nonumber\\[2mm]
                 =& -W_1^T(0,\hat{x},t) R_-\Lambda_-R_-^T W_1(0,\hat{x},t) \nonumber\\[2mm]
                 \geq & ~c_1W_1^T(0,\hat{x},t) R_- R_-^T W_1(0,\hat{x},t) \nonumber\\[2mm]
                 =& ~c_1|R_1^TW_1(0,\hat{x},t)|^2\label{5.15}
\end{align}
with $c_1$ a positive constant.
Applying Gronwall's inequality to \eqref{5.14}, we have 
\begin{align*}
 &\max_{t\in[0,T]}\|W_1(\cdot,\cdot,t)\|_{L^2(\mathbb{R}^+\times \mathbb{R}^{d-1})}^2 
\leq Ce^{CT}\bigg( \epsilon\|E_{2}\|^2_{L^2([0,T] \times \mathbb{R}^+\times \mathbb{R}^{d-1})} + \|E_{1}\|^2_{L^2([0,T] \times \mathbb{R}^+\times \mathbb{R}^{d-1})} \bigg). 
\end{align*}
At last, we use \eqref{5.15} and integrate \eqref{5.14} over $t\in[0,T]$ to obtain
\begin{align*}
& \|R_1^TW_1|_{x_1=0}\|_{L^2([0,T]\times \mathbb{R}^{d-1})}^2 
\leq C(T) \bigg(  \epsilon\|E_{2}\|^2_{L^2([0,T] \times \mathbb{R}^+\times \mathbb{R}^{d-1})} + \|E_{1}\|^2_{L^2([0,T] \times \mathbb{R}^+\times \mathbb{R}^{d-1})} \bigg). 
\end{align*}
This completes the proof of Lemma \ref{lemma5.1}.
\end{proof}

\begin{proof}[Proof of Lemma \ref{lemma5.2}]
By Remark \ref{remark3.1}, the Uniform Kreiss Condition is implied by the characteristic GKC. Thus the existence theory in \cite{BS} indicates that there exists a unique solution $W_2\in C([0,T];L^2(\mathbb{R}^+\times \mathbb{R}^{d-1}))$.

Next we adapt the method in \cite{GKO} to obtain the estimate. Define the Fourier transform of the solution to \eqref{5.11} with respect to $\hat{x}$:
$$
\widetilde{W}_2(x_1,\omega,t)=\int_{\mathbb{R}^{d-1}}e^{i \omega \hat{x}}W_2(x_1,\hat{x},t)d\hat{x}, \qquad \omega=(\omega_2,\omega_3,...,\omega_d)\in \mathbb{R}^{d-1}
$$
and the Laplace transform with respect to $t$:
$$
\widehat{W}_2(x_1,\omega,\xi)=\int_0^{\infty}e^{-\xi t}\widetilde{W}_2(x_1,\omega,t)dt,\qquad Re \xi>0. 
$$
Then we deduce from \eqref{5.11} that 
\begin{align}\label{5.16}
\left\{
{\begin{array}{*{20}l}
\vspace{2mm}A_1\partial_{x_1}\widehat{W}_2 = \left(\eta Q - \xi I_n -  i \sum_{j=2}^d\omega_j A_j\right)\widehat{W}_2,\\[2mm]
\vspace{2mm} B\widehat{W}_2(0,\omega,\xi)=\widehat{g}(\omega,\xi)-B\widehat{W}_1(0,\omega,\xi),\\[2mm]
\|\widehat{W}_2(\cdot,\omega,\xi)\|_{L^2(\mathbb{R}^+)}<\infty \quad \text{for ~a.e.} ~~\xi,~\omega. 
\end{array}}
\right.
\end{align}
Here $\eta = 1/\epsilon$. 
From Section \ref{sec3}, the differential equations in \eqref{5.16} can be reduced to
\begin{align}
    \partial_{x_1}(R_1^T\widehat{W}_2) =&~ M(\xi,\omega,\eta)(R_1^T\widehat{W}_2),\label{R1W2-ODE}\\[2mm]
    R_0^T\widehat{W}_2 =&~ -G_{00}^{-1}(\xi,\omega,\eta)G_{01}(\xi,\omega,\eta)(R_1^T\widehat{W}_2).
\end{align}
Since $BR_0=0$ and $I_n=R_0^TR_0+R_1^TR_1$, the boundary condition can be written as
\begin{equation}\label{Laplace-BC1}
    (BR_1)(R_1^T\widehat{W}_2)(0,\omega,\xi)=\widehat{g}(\omega,\xi)-(BR_1)(R_1^T\widehat{W}_1)(0,\omega,\xi).
\end{equation}

Let $R_M^S=R_M^S(\xi,\omega,\eta)$ and $R_M^U=R_M^U(\xi,\omega,\eta)$ be the respective right-stable and right-unstable matrices of $M=M(\xi,\omega,\eta)$:
\begin{align*}
MR_M^S=R_M^SM^S,\quad\quad 
MR_M^U=R_M^UM^U,
\end{align*}
where $M^S$ is a stable-matrix and $M^U$ is an unstable-matrix. 
In view of the Schur decomposition, we may choose $R_M^S$ and $R_M^U$ such that
$$
\begin{pmatrix}
R_M^{S*} \\[2mm]
R_M^{U*}    
\end{pmatrix}
\begin{pmatrix}
R_M^{S} & 
R_M^{U}    
\end{pmatrix}
=I_n.
$$
Recall that the superscript $*$ denotes the conjugate transpose.
Then from \eqref{R1W2-ODE} we obtain
\begin{align*}
\begin{pmatrix}
R_M^{S*} \\[2mm]
R_M^{U*}    
\end{pmatrix}\partial_{x_1}(R_1^T\widehat{W}_2)= 
\begin{pmatrix}
M^S & \\[2mm]
    & M^U 
\end{pmatrix}
\begin{pmatrix}
R_M^{S*} \\[2mm]
R_M^{U*}    
\end{pmatrix}
(R_1^T\widehat{W}_2).
\end{align*}
Since $\|\widehat{W}_2(\cdot,\omega,\xi)\|_{L^2(\mathbb{R}^+)}<\infty$ for a.e. $(\xi,\omega)$ and $M^U$ is an unstable-matrix, it must be
\begin{equation*}
   R_M^{U*}(R_1^T\widehat{W}_2)=0.
\end{equation*}
Thus the boundary condition in \eqref{Laplace-BC1} becomes
\begin{align*}
BR_1(R_1^T\widehat{W}_2)(0,\omega,\xi)= BR_1R_M^{S}R_M^{S*}(R_1^T\widehat{W}_2)(0,\omega,\xi)=&~\widehat{g}(\omega,\xi)-BR_1(R_1^T\widehat{W}_1)(0,\omega,\xi). 
\end{align*}
Since the matrix $(BR_1R_M^S)^{-1}$ is uniformly bounded due to the characteristic GKC, we conclude that
\begin{align*}
\left|(R_1^T\widehat{W}_2)(0,\omega,\xi)\right|^2=& \left|R_M^{S}R_M^{S*}(R_1^T\widehat{W}_2)(0,\omega,\xi)\right|^2\\[2mm]
=&\left|R_M^{S}(BR_1R_M^S)^{-1}\left[\widehat{g}(\omega,\xi)-BR_1(R_1^T\widehat{W}_1)(0,\omega,\xi)\right]\right|^2\\[2mm]
\leq & ~C\left(|\widehat{g}(\omega,\xi)|^2+|(R_1^T\widehat{W}_1)(0,\omega,\xi)|^2\right).
\end{align*}

By Parseval's identity, the last inequality leads to 
\begin{align*} 
&\int_{\mathbb{R}^{d-1}}\int_0^{\infty}e^{-2t Re\xi }|(R_1^TW_2)(0,\hat{x},t)|^2dtd\hat{x} \\[2mm]
\leq&~  C\bigg(\int_{\mathbb{R}^{d-1}}\int_0^{\infty}e^{-2t Re\xi }\left|g(\hat{x},t)\right|^2dtd\hat{x}+\int_{\mathbb{R}^{d-1}}\int_0^{\infty}e^{-2t Re\xi }\left|(R_1^TW_1)(0,\hat{x},t)\right|^2dtd\hat{x}\bigg)\\[2mm]
\leq&~ C\bigg(\int_{\mathbb{R}^{d-1}}\int_0^{\infty} \left|g(\hat{x},t)\right|^2dtd\hat{x}+\int_{\mathbb{R}^{d-1}}\int_0^{\infty} \left|(R_1^TW_1)(0,\hat{x},t)\right|^2dtd\hat{x}\bigg)\quad\text{for}~~Re\xi>0.
\end{align*}
Because the right-hand side is independent of $Re\xi$, we have 
\begin{align*}
&~\int_{\mathbb{R}^{d-1}}\int_0^{\infty} |(R_1^TW_2)(0,\hat{x},t)|^2dtd\hat{x} \\[2mm]
\leq&~ C\bigg(\int_{\mathbb{R}^{d-1}}\int_0^{\infty} \left|g(\hat{x},t)\right|^2dtd\hat{x}+\int_{\mathbb{R}^{d-1}}\int_0^{\infty} \left|(R_1^TW_1)(0,\hat{x},t)\right|^2dtd\hat{x}\bigg).
\end{align*}
By using the trick from \cite{GKO}, the integral interval $[0,\infty)$ in the last inequality can be changed to $[0,T]$:
\begin{align}
\|R_1^TW_2|_{x_1=0}\|_{L^2([0,T]\times \mathbb{R}^{d-1})}^2
\leq & ~C \bigg(\|R_1^TW_1|_{x_1=0}\|_{L^2([0,T]\times \mathbb{R}^{d-1})}^2+\|g\|_{L^2([0,T]\times \mathbb{R}^{d-1})}^2\bigg). \label{5.17}
\end{align}

At last, we multiply the equation in \eqref{5.11} with $W_2^T$ from the left to get
\begin{align*}
\frac{d}{dt}(W_2^TW_{2})+\partial_{x_1}(W_2^TA_1
 W_2)+\sum\limits_{j=2}^d \partial_{x_j}(W_2^TA_j
 W_2)\leq 0.
\end{align*}
Integrating the last inequality over $(x_1,\hat{x})\in[0,+\infty)\times \mathbb{R}^{d-1}$ and $t\in[0,T]$ yields
\begin{align*}
\max_{t\in[0,T]}\|W_2(\cdot,\cdot,t)\|_{L^2(\mathbb{R}^+\times \mathbb{R}^{d-1})}^2\leq &~ C\int_{\mathbb{R}^{d-1}}\int_0^{T} \left|(W_2^TA_1W_2)(0,\hat{x},t)\right|^2dtd\hat{x} \\[2mm]
= &~ C\int_{\mathbb{R}^{d-1}}\int_0^{T} \left|(W_2^TR_1)(R_1^TA_1R_1)(R_1^TW_2)(0,\hat{x},t)\right|^2dtd\hat{x} \\[2mm]
\leq &~ C\int_{\mathbb{R}^{d-1}}\int_0^{T} |(R_1^TW_2)(0,\hat{x},t)|^2dtd\hat{x}.
\end{align*} 
This together with \eqref{5.17} completes the proof.
\end{proof}

\section*{\normalsize{Acknowledgments}} 
The first author is funded by Alexander von Humboldt Foundation (Humboldt Research Fellowship Programme for Postdocs). The second author is supported by National Key Research and Development Program of China (Grant 2021YFA0719200) and National Natural Science Foundation of China (Grant 12071246).

\end{document}